\documentclass[11pt]{amsart} 
\usepackage{amsthm,amsbsy,amsmath,amssymb,amscd,amsfonts,array,mathrsfs,verbatim,enumerate,xypic,enumitem}
\usepackage[all]{xy}
\xyoption{arc} 

\newtheorem{thm}{Theorem}
\newtheorem{prop}[thm]{Proposition}     
\newtheorem{lem}[thm]{Lemma}
\newtheorem*{descentlemma}{Descent Lemma}

\theoremstyle{definition}

\newtheorem{rem}{Remark}[thm]

\DeclareFontFamily{OT1}{rsfs}{}
\DeclareFontShape{OT1}{rsfs}{n}{it}{<-> rsfs10}{}
\DeclareMathAlphabet{\curly}{OT1}{rsfs}{n}{it}

\newcommand{\C}{{\bf C}} 
\newcommand{\Sets}{{\bf Sets}} 
\newcommand{\Ab}{{\bf Ab}}   
\newcommand{\Sch}{{\bf Sch}} 
\newcommand{\Mon}{{\bf Mon}} 

\newcommand{\D}{{\bf D}} 
\newcommand{\Cat}{{\bf Cat}}
\newcommand{\CFG}{{\bf CFG}}
\newcommand{\LogCat}{{\bf LogCat}}
\newcommand{\LogCFG}{{\bf LogCFG}}
\newcommand{\Log}{{\bf Log}}   
\newcommand{\LogSch}{{\bf LogSch}} 

\renewcommand{\a}{\mathfrak{a}}
\renewcommand{\AA}{\mathbb{A}}  
\newcommand{\ZZ}{\mathbb{Z}} 
\newcommand{\FF}{\mathbb{F}} 
\newcommand{\NN}{\mathbb{N}} 

\newcommand{\m}{\mathfrak{m}}         

\newcommand{\M}{\mathcal{M}} 
\newcommand{\N}{\mathcal{N}} 
\newcommand{\X}{\curly{X}} 
\newcommand{\Y}{\curly{Y}} 
\newcommand{\Z}{\curly{Z}} 
\newcommand{\W}{\curly{W}} 
\renewcommand{\O}{\mathcal O} 

\renewcommand{\u}{\underline}

\newcommand{\ov}{\overline}
\newcommand{\into}{\hookrightarrow}
\newcommand{\be}{\begin{eqnarray*}}
\newcommand{\ee}{\end{eqnarray*}}
\newcommand{\bne}[1]{\begin{eqnarray} \label{#1} }
\newcommand{\ene}{\end{eqnarray}}
\newcommand{\xym}{\xymatrix}
\newcommand{\bp}{\begin{pmatrix}}
\newcommand{\ep}{\end{pmatrix}}

\newcommand{\dlog}{\operatorname{dlog}}
\newcommand{\Dom}{\operatorname{Dom}}  

\newcommand{\Hom}{\operatorname{Hom}}

\newcommand{\Aut}{\operatorname{Aut}}

\newcommand{\Cok}{\operatorname{Cok}}

\newcommand{\Id}{\operatorname{Id}}
\newcommand{\Spec}{\operatorname{Spec}}

\begin{document}

\author{W.~D.~Gillam}
\address{Department of Mathematics, Brown University, Providence, RI 02912}
\email{wgillam@math.brown.edu}
\date{\today}
\title{Logarithmic stacks and minimality}

\begin{abstract} Given a category fibered in groupoids over schemes with a log structure, one produces a category fibered in groupoids over log schemes.  We classify the groupoid fibrations over log schemes that arise in this manner in terms of a categorical notion of ``minimal" objects.  The classification is actually a purely category-theoretic result about groupoid fibrations over fibered categories, though most of the known applications occur in the setting of log geometry, where our categorical framework encompasses many notions of ``minimality" previously extant in the literature.   \end{abstract}

\maketitle

\section{Introduction} \label{section:intro} We assume for the sake of brevity that the reader has some familiarity with fibered categories, though all necessary definitions can be found in \S\ref{section:fiberedcategory}, which the reader may refer to as necessary.  The results of this paper are purely category-theoretic.  The proofs are not particularly difficult or enlightening, and will be relegated to \S\ref{section:proofs}.

Throughout, we consider a fibered category \be \LogSch & \to & \Sch \\ X & \mapsto & \u{X} \\ (f : X \to Y) & \mapsto & (\u{f} : \u{X} \to \u{Y}) \ee with associated groupoid fibration $\Log \to \Sch$ (the restriction of $X \mapsto \u{X}$ to the subcategory of cartesian morphisms).  The underline notation will be used liberally: an underlined object is always one of $\Sch$ and for a functor $F$ to $\Log$ or $\LogSch$, $\u{F}$ is the functor to $\Sch$ defined by composing $F$ with $X \mapsto \u{X}$.

The main example we have in mind is where $\LogSch$ is the category of (integral, or fine, or \dots) log schemes (\S\ref{section:examples}), and $X \to \u{X}$ is the forgetful functor to schemes taking a log scheme to the underlying scheme.  In this situation, a cartesian morphism of log schemes $f : X \to Y$ is the same thing as a \emph{strict} morphism, meaning $f^{\dagger} : f^*\M_Y \to \M_X$ is an isomorphism. 

A \emph{\Log~structure} on a functor $F : \X \to \Sch$ is a functor $\M : \X \to \Log$ with $F = \u{\M}$.\footnote{This agrees with F.~Kato's definition [FK, 3.1] in the situation of the ``main example" and it agrees with the usual notion of a log structure on a scheme $X$ in the case $\X = \Sch / X$.}  Let $\LogCat / \Sch$ denote the $2$-category of categories over $\Sch$ with $\Log$ structure.  Precise definitions of $1$- and $2$-morphisms in $\LogCat / \Sch$ are given in \S\ref{section:logstructures}.

We define a morphism of $2$-categories \be \Phi : \LogCat / \Sch & \to & \Cat / \LogSch \\ (\M : \X \to \Log) & \mapsto & (\X,\M) \ee as follows (c.f.\ [FK, 3.4]).  The objects of $(\X,\M)$ are pairs $(x,f : X \to \M x)$ where $x$ is an object of $\X$ and $f$ is a morphism in the fiber category $\LogSch_{\u{X}}$ (that is, we require $\u{X} = \u{\M}x$ and $\u{f} = \Id_{\u{X}}$).  A morphism $$ h = (a,b) : (x,f:X \to \M x) \to (y, g : Y \to \M y) $$ in $(\X,\M)$ is a pair consisting of a morphism $a : x \to y$ in $\X$ and a morphism $b : X \to Y$ in $\LogSch$ making the diagram in $\LogSch$ $$ \xym{ X \ar[r]^f \ar[d]_b & \M x \ar[d]^{\M a} \\ Y \ar[r]^g & \M y } $$ commute.  The behavior of $\Phi$ on $1$- and $2$-morphisms is explained carefully in \S\ref{section:phi}.

Let $ \LogCFG / \Sch$ denote the full subcategory of $\LogCat / \Sch$ whose objects are those $\M : \X \to \Log$ where $\u{\M} : \X \to \Sch$ is a groupoid fibration.  One easily proves (c.f.\ \S\ref{section:philemma}):

\begin{prop} \label{prop:phi} The functor $\Phi$ takes $\LogCFG / \Sch$ into $\CFG / \LogSch$, hence defines a functor \be \Phi_{\CFG} : \LogCFG / \Sch & \to & \CFG / \LogSch \\ (\M : \X \to \Log) & \mapsto & (\X,\M). \ee \end{prop}

Our main result is a description of the essential image of $\Phi_{\CFG}$ (c.f.\ [FK, 3.5]) in terms of minimal objects, which we now define.

\noindent {\bf Definition.}  For a functor $F : \Z \to \LogSch$, we say that an object $z \in \Z$ is \emph{minimal} iff any solid $\Z$ diagram $$ \xym{ w \ar@{.>}[rr]^k & & z \\ & \ar[lu]^i \ar[ru]_j w' } $$ with $\u{F}i = \u{F}j = \Id$ has a unique completion $k$.

For $(\M : \X \to \Log) \in \LogCFG / \Sch$, an object $z=(x, f : X \to \M x)$ of $(\X , \M)$ is minimal iff $f$ is an isomorphism (\S\ref{section:minimalobjects}).  Indeed, we arrived at the definition of minimal objects by abstracting the important category theoretic properties of such objects $z$.  The main result, proved in \S\ref{section:criterion}, is the

\begin{descentlemma}  A groupoid fibration $F : \Z \to \LogSch$ is in the essential image of $\Phi_{\CFG}$ iff it satisfies the following conditions: \begin{enumerate}[label=B\theenumi., ref=B\theenumi]   \item \label{enoughminimals} For every $w \in \Z$, there is a minimal object $z \in \Z$ and a $\Z$ morphism $i : w \to z$ with $\u{F}i = \Id$.  \item \label{strict} For any $\Z$ morphism $i : w \to z$ with $z$ minimal, $Fi$ is cartesian iff $w$ is minimal. \end{enumerate} In fact, assuming the conditions are satisifed, the restriction $\M$ of $F$ to the full subcategory $\Z^m$ of minimal objects defines an object $(\M : \Z^m \to \Log)$ of $\LogCFG / \Sch$ and there is an equivalence $\Z \simeq (\Z^m , \M)$ in $\CFG / \LogSch$. \end{descentlemma}

One might say that $F$ has \emph{enough minimals} when $F$ satisfies \eqref{enoughminimals} and $F$ has \emph{strict/minimal compatibility} when $F$ satisfies \eqref{strict}, though we do not use this terminology here.

Several instances of this result can be found in the literature on log algebraic stacks, typically in the following form: One defines a groupoid fibration $\Z$ over the category of (fine, or fs) log schemes, then one manages to pick some class of ``minimal objects"  (often called ``basic objects") in $\Z$ by ``pure thought."  Then one verifies the necessary hypotheses and essentially reproves the above theorem for the particular $\Z$ in question.  The ``pure thought" step can now be subsumed by our general notion of minimality, though it is still often necessary to translate our categorical definition into something concrete for a particular $\Z$ in order to check \eqref{enoughminimals},\eqref{strict}.  In log geometry, the concrete notion of minimality is typically a condition on characteristic monoids of various log structures.  In \S\ref{section:examples}, we discuss some specific examples, including: the \emph{log curves} of \cite{FK} (\S\ref{section:logcurves}), and the \emph{log points} of \cite{ACGM} (\S\ref{section:logpoints}).  It can also be shown (directly, without appeal to representability results and the Descent Lemma) that the ``basic/minimal" notion for the \emph{log stable maps} of \cite{C}, \cite{AC}, \cite{GS} is equivalent to our notion.

Our result may be viewed as the first step toward a Logarithmic Artin Criterion for representability of a groupoid fibration over log schemes by an Artin stack with fine log structure.  The problem solved in this paper is posed explicitly by Olsson in [O, 3.5.3].

\noindent {\bf Acknowledgements.}  I thank Dan Abramovich for suggesting the problem of finding a categorical notion of minimality and for suggesting the fibered categories setup.  Qile Chen and Steffen Marcus also provided useful comments on a preliminary version of this work.  This research was partially funded by an NSF Postdoctoral Fellowship.

\section{Proofs} \label{section:proofs}  This section contains the definitions, technical details, and proofs not fully presented in the Introduction.

\subsection{Definitions} \label{section:fiberedcategory} ([Vis 3.1]\footnote{We use Vistoli's definition throughout, but keep in mind that it differs from Definition 5.1 in [SGA1.VI]}) Let $F : \C \to \D$ be a functor.  A $\C$ morphism $f : c' \to c$ is called \emph{cartesian} (relative to $F$) iff, for any $\C$ morphism $c'' \to c'$ the map $h \mapsto Fh$ yields a bijection from completions of the left diagram below to completions of the right diagram below.  $$\xym{ c'' \ar@{.>}[rr]^h \ar[rd] & & c' \ar[ld]^f \\ & c} \quad \quad \xym{ Fc'' \ar@{.>}[rr] \ar[rd] & & Fc' \ar[ld]^{Ff} \\ & Fc } $$  The functor $F$ is called a \emph{fibered category} iff, for any $\D$ morphism $f : d \to d'$ and any object $c'$ of $\C$ with $Fc'=d'$, there is a cartesian morphism $\ov{f} : c \to c'$ with $F \ov{f}=f$.  A fibered category $F : \C \to \D$ is called a \emph{groupoid fibration} iff every $\C$ morphism is cartesian.  A \emph{morphism} of fibered categories $$G : (F : \C \to \D) \to (F' : \C' \to \D)$$ over $\D$ is a functor $G : \C \to \C'$ satisfying $F'G=F$ (actual equality, not merely equivalence) and taking cartesian arrows to cartesian arrows.  

Every $\C$ isomorphism is cartesian, and cartesian morphisms enjoy the following \emph{2-out-of-3 property}: given $\C$ morphisms $f : c \to c'$, $g : c' \to c''$ with $g$ cartesian, one sees easily that $f$ is cartesian iff $gf$ is cartesian.  In particular, a composition of cartesian morphisms is cartesian.  If $F$ is a fibered category, then the restriction of $F$ to the subcategory of $\C$ with the same objects but with only cartesian morphisms as morphisms is a groupoid fibration, called the \emph{groupoid fibration associated to} $F$.

\subsection{More on \Log~structures} \label{section:logstructures} Recall from \S\ref{section:intro} that a \emph{\Log ~structure} on a category $F : \X \to \Sch$ over $\Sch$ is a functor $\M : \X \to \Log$ with $F = \u{\M}$.  Categories over $\Sch$ with \Log~structure form a $2$-category, which has the same objects as $\Cat / \Log$, though there is some ambiguity about what the $1$- and $2$-morphisms are in $\Cat / \Log$, so to be absolutely clear, a $1$-morphism $$(F, F^\dagger) : (\M : \X \to \Log) \to (\N : \Y \to \Log) $$ in $\LogCat / \Sch$ is a functor $F : \X \to \Y$ with $\u{\N F} = \u{\M}$ (i.e.\ a $1$-morphism in $\Cat / \Sch$ between the underlying categories over $\Sch$) together with a natural transformation $$ F^\dagger : \M \to \N F$$ of functors $\X \to \LogSch$ with $\underline{F^\dagger} = \Id$.  The functors in question are naturally viewed as functors to $\Log$, but we do not require the $\LogSch$ morphisms $$F^\dagger(x) : \M x \to \N F x$$ to be cartesian (though we do require $\u{F^\dagger(x)} = \Id$).  If, however, $F^\dagger(x)$ is cartesian for every $x \in \X$, then $(F,F^\dagger)$ is called \emph{cartesian} (or perhaps \emph{strict}).  A $2$-morphism $\eta : (F,F^\dagger) \to (G,G^\dagger)$ is a natural transformation $\eta : F \to G$ with $\u{\eta} = \Id$ making the diagram of natural transformations $$ \xym{ & \M \ar[ld]_{F^\dagger} \ar[rd]^{G^\dagger} \\ \N F  \ar[rr]^{\N \eta} & & \N G   } $$ commute. 

\subsection{More about $\Phi$} \label{section:phi}  For a category $\M : \X \to \Log$ over $\Sch$ with $\Log$ structure, recall that we defined a category $(\X,\M)$ whose objects are pairs $(x,f : X \to \M x)$ where $x$ is an object of $\X$ and $f$ is a morphism in $\LogSch_{\u{X}}$.  The category $(\X,\M)$ is viewed as a category over $\LogSch$ via the forgetful functor \be F : (\X, \M) & \to & \LogSch \\ (x, f : X \to \M x) & \mapsto & X. \ee  This construction defines a morphism \be \Phi : \LogCat / \Sch & \to & \Cat / \LogSch \\ (\M : \X \to \Log) & \mapsto & (\X,\M) \ee of $2$-categories defined on $1$- and $2$-morphisms as follows.

For a $1$-morphism $$(F, F^\dagger) : (\M : \X \to \Log) \to (\N : \Y \to \Log) ,$$ we define \be \Phi(F,F^\dagger) : (\X , \M) & \to & (\Y, \N) \\ (x, f: X \to \M x) & \mapsto & (Fx, F^\dagger(x) \circ f : X \to \N Fx ) . \ee  For a $2$-morphism $\eta : (F,F^\dagger) \to (G,G^\dagger) $ we define the $2$-morphism $$\Phi(\eta) : \Phi(F,F^\dagger) \to \Phi(G,G^\dagger) $$ by declaring $$\Phi(\eta)(x,f:X \to \M x) : \Phi(F,F^\dagger)(x,f:X \to \M x) \to \Phi(G,G^\dagger)(x, f: X \to \M x) $$ to be the $(\Y, \N)$ morphism $$ (\Id_X, \eta(x)) : (Fx, F^\dagger(x) \circ f : X \to \N F x) \to (Gx, G^\dagger x \circ f : X \to \N G x). $$  Note that the diagram $$ \xym@C+10pt{ X \ar@{=}[d] \ar[r]^f & \M x \ar@{=}[d] \ar[r]^{F^\dagger x } & \N Fx \ar[d]^{\N \eta(x)} \\ X \ar[r]^f & \M x \ar[r]^{G^\dagger x} & \N G x } $$ commutes by the commutativity condition on the $\LogCat / \Sch$ $2$-morphism $\eta$.

By definition, a $2$-morphism in $\Cat / \LogSch$ between $1$-morphisms $$\phi, \psi : (F : \Z \to \LogSch) \to (G : \W \to \LogSch)$$ is a natural transformation $\eta : \phi \to \psi$ such that the $\W$ morphism $\eta(z) : \phi(z) \to \psi(z)$ satisfies $G \eta(z) = \Id$ for every $z \in \Z$.  (Note that our functor $\Phi$ takes $2$-morphisms in $\LogCat / \Sch$ to such $2$-morphisms!)  This is stronger than only requiring $\u{G} \eta(z) = \Id$.  This is worth emphasizing because we will be interested in the essential image of morphisms to $\Cat / \LogSch$ and the notion of essential image depends on the notion of $2$-morphisms.

\subsection{Proof of Proposition~\ref{prop:phi}} \label{section:philemma} Given a morphism $b : X \to Y$ in $\LogSch$ and an object $(y, g : Y \to \M y)$ of $(\X, \M)$ lying over $Y$, we use the fact that $\u{\M} : \X \to \Sch$ is a groupoid fibration (and the fact that $y \in \X_{\u{Y}}$ by definition of the objects of $(\X,\M)$) to find an $\X$ morphism $a : x \to y$ such that $\u{\M} a = \u{b}$, then we use the fact that the $\Log$ morphism $\M a$ is a cartesian morphism in $\LogSch$ to find a $\LogSch$ morphism $f : X \to \M x$ with $\u{f}= \Id_{\u{X}}$ making the diagram $$ \xym{ X \ar[r]^f \ar[d]_b & \M x \ar[d]^{\M a} \\ Y \ar[r]^g & \M y } $$ commute.  We have constructed an $(\X, \M)$ morphism \be (a,b) : (x,f : X \to \M x) & \to & (y,g : Y \to \M y) \ee lying over $b$ with the desired codomain. 

Given two lifts \be (a,b) : (x,f : X \to \M x) & \to & (y,g : Y \to \M y) \\ (a',b) : (x',f' : X \to \M x') & \to & (y,g : Y \to \M y) \ee of $b$ with codomain $(y, g : Y \to \M y)$, we use the fact that $\X \to \Sch$ is a groupoid fibration to find a unique $h : x \to x'$ such that $\u{h} = \Id$ and $a = a' h$.  The commutativity $f' = (\M h)f$ is automatic by the uniqueness requirement in the definition of a cartesian morphism, so \be (h,\Id_X) : (x, f : X \to \M x) & \to & (x',f' : X \to \M x') \ee is the unique lift of $\Id_X$ to $(\X,\M)$ satisfying $(a,b) = (a', b) (h,\Id_X)$. 

\subsection{Minimal objects} \label{section:minimalobjects} We now collect together some basic properties of minimal objects for later use.

\noindent {\bf Definition.} An object $D$ of a category $\C$ is called \emph{pseudo-terminal} iff $\Hom_{\C}(C,D)$ is a pseudo-torsor (either empty or a torsor) under $\Aut_{\C}(D)$ for every $C \in \C$.  

\begin{prop} \label{prop:minimalityproperties} Let $F : \Z \to \LogSch$ be any functor.  If $z \in \Z$ is minimal, then $z$ is pseudo-terminal in its fiber category $\Z_{\u{F}z}$ over $\Sch$ and any map $z \to w$ in $\Z_{\u{F}z}$ has a retract in $\Z_{\u{F}z}$ (and $z \to w$ is an isomorphism when $w$ is also minimal).\end{prop}

\begin{proof} For the first statement: Given $f,g : w \to z$ with $\u{F}f = \u{F} g = \Id$, use minimality of $z$ to uniquely complete $$\xym{  z \ar@{.>}[rr]^k  & & z \\ & w \ar[lu]^f \ar[ru]_g } $$ and note that $\u{F} k =\Id$ by applying $\u{F}$ to the completed diagram.  Note that $k$ is an isomorphism because one can find an inverse by exchanging the roles of $f,g$ above.

For the second statement, if $f : z \to w$ has $\u{F}f = \Id$, then we obtain a retract $r$ of $f$ in $\Z_{\u{F}z}$ by completing the diagram $$\xym{  w \ar@{.>}[rr]^r  & & z \\ & z \ar[lu]^f \ar@{=}[ru] } $$ using minimality of $z$.  To prove that the retract $r : w \to z$ is an inverse for $f$, it suffices to prove that $fr = \Id$.  But when $w$ is minimal, $r$ also has a retract, so in particular it is monic, hence it suffices to prove $rfr = r \Id,$ which is obvious because $r$ retracts $f$. \end{proof}

\noindent {\bf Definition.}  A set $\curly{W}$ of objects of a category $\C$ is called \emph{weakly terminal} iff it satisfies the following two conditions: \begin{enumerate}[label=W\theenumi., ref=W\theenumi] \item \label{W1} For every $C \in \C$ there is a $D \in \curly{W}$ and a $\C$ morphism $C \to D$. \item \label{W2} For any solid $\C$ diagram $$\xym{  D \ar@{.>}[rr]^f  & & D' \\ & C \ar[lu] \ar[ru] } $$ with $D,D' \in \curly{W}$ there is a unique $f$ completing the diagram. \end{enumerate}  An object $D$ of $\C$ is called \emph{weakly terminal} iff $\{ D \}$ is weakly terminal in the above sense.  That is, $D$ is weakly terminal iff $\Hom_{\C}(C,D)$ is a torsor under $\Aut_{\C}(D)$ for every $C \in \C$.

\noindent {\bf Remark.}  The $f$ in \eqref{W2} is always an isomorphism by the usual argument: reverse the roles of $D,D'$ to get an $h : D' \to D$, then show that $fh$ and $hf$ are the identity by using uniqueness of the completion.

\begin{prop} If a functor $F : \Z \to \LogSch$ satisfies \eqref{enoughminimals}, then for any $z \in \Z$ the set $\Z_{\u{F}z} \cap \Z^{\rm min}$ of minimal objects in the fiber category over $\Sch$ is weakly terminal in this fiber category. \end{prop}

\begin{proof} Certainly \eqref{enoughminimals} implies that $\Z_{\u{F}z} \cap \Z^{\rm min}$  has property \eqref{W1}.  To see that $\Z_{\u{F}z} \cap \Z^{\rm min}$  has property \eqref{W2}, consider a solid diagram $$ \xym{ z' \ar@{.>}[rr]^k & & z'' \\ & w \ar[lu]^f \ar[ru]_g  } $$ in $\Z_{\u{F}z}$ with $z',z''$ minimal.  The unique completion $k$ exists by minimality for $z''$.  \end{proof}

\subsection{Proof of the Descent Lemma} \label{section:criterion}  We start by proving a lemma which is basically one implication in the Descent Lemma.

\begin{lem} \label{lem:thecriterion} For $(\M : \X \to \Log) \in \LogCFG / \Sch$, an object $z=(x, f : X \to \M x)$ of $(\X , \M)$ is minimal iff $f$ is an isomorphism.  For every object $w$ of $(\X,\M)$, there is a minimal object $z$ of $(\X,\M)$ and a morphism $i : w \to z$ with $\u{F} i = \Id$.   For every minimal object $z = (x,f : X \to \M x)$ of $(\X,\M)$, every $w = (y,g : Y \to \M y)$ in $(\X,\M)$, and every morphism $i = (a,b) : w \to z$, the map $b : Y \to X$ is cartesian iff $w$ is minimal.  \end{lem}

\begin{proof} To prove that $z=(x, f : X \to \M x)$ is minimal when $f$ is an isomorphism, first note that the solid part of a diagram $$ \xym{ w \ar@{.>}[rr]^k & & z \\ & w' \ar[lu]^i \ar[ru]_j  } $$ in $(\X, \M)$ with $\u{F}i = \u{F}j = \Id$  corresponds to a solid diagram $$ \xym{ y \ar@{.>}[rr]^c & & x \\ & y' \ar[lu]^a \ar[ru]_b } $$ in $\X$ with $\u{\M} a = \u{\M} b = \Id$ together with a solid diagram $$ \xym@R-15pt{ Y \ar@{.>}[rr]^d \ar[dd] & & X \ar[dd]^f_{\cong} \\ & Y' \ar[lu] \ar[ru] \ar[dd] \\ \M y \ar@{.>}[rr] & & \M x \\ & \M y' \ar[lu]^{\M a} \ar[ru]_{\M b} } $$ in $\LogSch$ where the underlined versions of all vertical arrows are $\Id$.  Since $\u{\M}$ is a groupoid fibration and $\u{\M} a = \Id$, $a$ is an isomorphism, so the first diagram can be uniquely completed with an $\X$ morphism $c : y \to x$ as indicated. Then, after filling in the bottom dotted arrow in the second diagram with $\M c$, there is a unique choice $d$ for the top dotted arrow because $f$ is an isomorphism.  (One also checks easily that this $d$ makes the top triangle commute.)  Evidently $k := (c,d)$ is the unique completion of the original diagram.  

To prove that $f$ is an isomorphism when $z = (x,f : X \to \M x)$ is minimal, consider the map $$(\Id_x, f) : (x, f : X \to \M x) \to (x, \Id : \M x \to \M x). $$ Since $(x, \Id : \M x \to \M x)$ is minimal by what we just proved and $z$ is minimal, the map $(\Id_x,f)$ must be an isomorphism by the last part of Proposition~\ref{prop:minimalityproperties}, hence $f$ must be an isomorphism.

For the next statement, if $w=(x, f: X \to \M_z)$, just take $z := (x, \Id : \M x \to \M x)$ and $i = (\Id_x, f)$.  The object $z$ is minimal by the statement we just proved. 

For the final statement, note that by definition of a morphism in $(\X,\M)$, we have a commutative diagram $$ \xym{ Y \ar[r]^g \ar[d]_b & \M y \ar[d]^{\M a} \\ X \ar[r]^f & \M x } $$ in $\LogSch$ and by what we have already proved, the minimality of $z$ means that $f$ is an isomorphism, so $f$ and its inverse $f^{-1}$ are cartesian (isomorphisms are cartesian, \S\ref{section:fiberedcategory}).  If $b$ is cartesian, then $g$ is also cartesian by the 2-out-of-3 property of cartesian morphisms  (\S\ref{section:fiberedcategory}) because $(\M a)g = fb$ is cartesian, and $\M a$ is cartesian; but then $g$ must be an isomorphism since $\Log \to \Sch$ is a groupoid fibration, so $w$ is minimal.  If $w$ is minimal, then $g$ is an isomorphism, hence cartesian, so $b =   f^{-1} (\M a)g$ is also cartesian since it is displayed as a composition of cartesian morphisms.

\end{proof}

Next we need two preparatory lemmas.  The first says that, under mild hypotheses on $F$, a minimal object actually enjoys a stronger lifting property than the defining one.

\begin{lem} \label{lem:minimality1} Suppose $F : \Z \to \LogSch$ is a groupoid fibration satisfying \eqref{strict}.  Then any solid diagram $$ \xym{ w \ar@{.>}[rr]^k & & z \\ & w' \ar[ul]^i \ar[ur]_j } $$ with $\u{F} i = \Id$ and $z$ minimal has a unique completion $k$. \end{lem}

\begin{proof} This would be clear from the definitions if we also knew that $\u{F}j = \Id$, so our goal is to reduce to that case.  We first use that $\LogSch \to \Sch$ is a fibered category to find a cartesian arrow $f : X \to F z$ with $\u{f} = \u{F} j$.  Then we use the property of the cartesian arrow $f$ to find a (unique) completion $h$ of $$ \xym{ Fw' \ar@{.>}[rr]^h \ar[rd] & & X \ar[ld] \\ & Fz } $$ in $\LogSch$ with $\u{h}= \Id$.  Next we use that $F$ is a groupoid fibration to find a $\Z$ morphism $g : z' \to z$ with $F g = f$ and a completion $j'$ of $$ \xym{ w' \ar@{.>}[rr]^{j'} \ar[rd]_j & & z' \ar[ld]^g \\ & z } $$ with $F j' = h$.  The object $z'$ is minimal since $z$ is minimal, $Fg = f$ is cartesian, and $F$ satisfies \eqref{strict}.  Since we also have $\u{F} j' = \u{h} = \Id$, the defining property of the minimal object $z'$ ensures that there is a unique completion $l$ of the diagram $$ \xym{ w \ar@{.>}[rr]^l & & z' \ar[rd]^g \\ & \ar[lu]^i \ar[ru]_{j'} w' \ar[rr]_j & & z } $$ so we can certainly complete the original diagram by taking $k = gl$.  To show that this is the \emph{unique} completion of the original diagram, it suffices to show that \emph{any} such completion $k$ can be factored through $g$.  That is, given any $k$ completing the original diagram we want to prove that there is a completion $a$ of the diagram: $$ \xym{ w \ar@{.>}[rr]^a \ar[rd]_k & & z' \ar[ld]^g \\ & z } $$  Certainly we can complete $\u{F}$ of this diagram in $\Sch$ with the identity map since $\u{F} i = \u{F} j' = \Id$.  We can then lift that (identity) completion to a completion of $F$ of this diagram using the fact that $Fg = f$ is cartesian.  We can then lift \emph{that} completion to the desired completion $a$ using that $F$ is a groupoid fibration. \end{proof}

\begin{rem} For a $(\X, \M : \X \to \Log) \in \LogCFG / \Sch$, our proof of Lemma~\ref{lem:thecriterion} actually shows that an object $z = (x,f : X \to \M x)$ of $(\X,\M)$ with $f$ an isomorphism has this stronger lifting property (we never used that $\u{\M} b = \Id$ in the first paragraph of the proof).  \end{rem}

\begin{lem} \label{lem:minimality2} Suppose $F : \Z \to \LogSch$ is a groupoid fibration satisfying \eqref{strict}.  Then any solid $\Z$ diagram $$ \xym{ z' \ar@{.>}[rr]^k \ar[rd]_i & & z \ar[ld]^j \\ & w }$$ with $\u{F}i = \u{F} j$ and $w,z$ minimal has a unique completion $k$ with $\u{k} = \Id$. \end{lem}

\begin{proof} If we also knew that $\u{F} j = \Id$, then this would be trivial because then $j$ would be an isomorphism by Proposition~\ref{prop:minimalityproperties}, so our goal is to reduce to that case as we did in the previous proof.  As usual, we use that $\LogSch \to \Sch$ is a fibered category to find a cartesian arrow $f : X \to Fw$ lying over the $\Sch$ morphism $\u{F} j$, then we use the defining property of $f$ to find (unique) $\LogSch$ morphisms $i', j'$ completing the diagram $$ \xym{ Fz' \ar@/_1pc/[rdd]_{Fi} \ar@{.>}[rd]^{i'} \ar@{.>}[rr] & & Fz \ar@{.>}[ld]_{j'} \ar@/^1pc/[ldd]^{Fj} \\ & X \ar[d]^f \\ & Fw }$$ with $\u{i}' = \u{j}' = \Id$.  Then we use that $F$ is a groupoid fibration to find $g : w' \to w$ with $Fg=f$ and $a,b$ completing $$ \xym{ z' \ar@/_1pc/[rdd]_{i} \ar@{.>}[rd]^{a} \ar@{.>}[rr]^k & & z \ar@{.>}[ld]_{b} \ar@/^1pc/[ldd]^{j} \\ & w' \ar[d]^g \\ & w }$$ with $Fa=i'$, $Fb=j'$.  We know $w'$ is minimal by \eqref{strict} since $Fg=f$ is cartesian and $w$ is minimal, and we have $\u{F}a = \u{F}b = \Id$, so $b$ is an isomorphism by Proposition~\ref{prop:minimalityproperties} because $z$ is also minimal, hence we certainly have a unique completion $k$ \emph{making the small triangle commute} (i.e. $a = bk$, which automatically implies $\u{F}k = \Id$), and this same $k$ automatically makes the big triangle commute (i.e. $i = jk$).  To complete the proof (i.e.\ to establish the uniqueness statement), it is enough to show that any $k'$ making the big triangle commute (i.e. $i = jk'$) is equal to $k$ \emph{under the additional hypothesis} $\u{F}k' = \Id$  (the statement is not generally true without this additional hypothesis).  Since $b$ is an isomorphism, so is $j'=Fb$, so in particular it is cartesian, hence $Fj = f j'$ is also cartesian since it is manifestly a composition of cartesian arrows.  Since $i = jk'$ implies $Fi = (Fj)(Fk')$ and $\u{F}k'=\Id$, we know $Fk = Fk'$ by the uniqueness property of the cartesian arrow $Fj$.  But then we can conclude $k = k'$ since $F$ is a groupoid fibration and we have: $i = jk$, $i = jk'$, $Fk = Fk'$.  \end{proof}

\noindent \emph{Proof of the Descent Lemma.} For the ``only if" implication, note that the conditions are invariant under equivalences in the $2$-category $\CFG / \LogSch$, so we just need to prove that they hold in $(\X, \M)$ for any $$(\M : \X \to \Log) \in \LogCFG / \Sch.$$  This is exactly what we proved in Lemma~\ref{lem:thecriterion}.  

For the ``if" implication, we suppose the two conditions hold, then proceed in five steps.

\noindent {\bf Step 1.}  We prove that $ \u{\M} = \u{F} : \Z^m \to \Sch$ is a groupoid fibration.  Given an $\Sch$ morphism $\u{f} : \u{Z} \to \u{W}$ and $w \in \Z^m$ with $\u{F}w=\u{W}$, first lift $\u{f}$ to a $\Log$ morphism $f : Z \to Fw$ using the fact that $\Log \to \Sch$ is a groupoid fibration, then lift $f$ to an $\Z$ morphism $\ov{f} : z \to w$ using the fact that $F$ is a groupoid fibration.  Then by \eqref{strict}, $z \in \Z^m$ because $w \in \Z^m$ and $F \ov{f} = f$ is cartesian.  The essential uniqueness of liftings of $\u{f}$ to $\Z^m$ with codomain $w$ is immediate from Lemma~\ref{lem:minimality2}.

\noindent {\bf Step 2.}  We construct a $\CFG / \LogSch$ 1-morphism $\psi : (\Z^m , \M = F) \to \Z$ as follows.  For each $\LogSch$ morphism $f : X' \to X$ with $\u{f} = \Id_{\u{X}}$, and each $x \in \Z^m_X$, choose some $\Z$ morphism $f_x : z_x \to x$ with $F f_x = f$ (such a choice, called a \emph{cleavage}, exists since $F$ is a groupoid fibration).  Now define $\psi$ on objects by \be \psi(x, f : X \to \M x) & \mapsto & z_x = \Dom f_x . \ee  Given a morphism $$(a,b) : (x,f:X \to \M x) \to (y, g : Y \to \M y), $$ define $\psi(a,b) : z_x \to z_y$ to be the unique $\Z$ morphism with $F \psi(a,b) = b$ and making the $\Z$ diagram $$ \xym{ z_x \ar[r]^{f_x} \ar@{.>}[d]_{\psi(a,b)} & x \ar[d]^a \\ z_y \ar[r]^{g_y} & y } $$ commute. (It is a well-known exercise with groupoid fibrations to see that there is a unique such morphism so that the resulting $\Z$ diagram lies over the diagram $$ \xym@C+15pt{ X  \ar[r]^-{f = Ff_x} \ar[d]_b & \M x = F x \ar[d]^{ \M a = Fa} \\ Y \ar[r]^-{g = F g_y} & \M y = F y } $$ in $\LogSch$.)  

\noindent {\bf Step 3.}  We construct a $\CFG / \LogSch$ 1-morphism $\phi : \Z \to (\Z^m , \M = F)$ as follows.  For each $z \in \Z$ choose a $\Z$ morphism $f^m_z : z \to x_z$ with $x_z \in \Z^m$ and $\u{F}f^m_z = \Id$ using \eqref{enoughminimals} and set $$\phi(z) := (x_z, F f^m_z : F z \to F x_z = \M x_z) .$$  To define $\phi$ on an $\Z$ morphism $h : z \to w$, we use Lemma~\ref{lem:minimality1} and the fact that $x_w \in \Z^m$ is minimal (and $\u{F} f^m_z = \Id$) to complete the diagram $$ \xym{ z \ar[r]^{f^m_z} \ar[d]_h & x_z \ar@{.>}[d]^k \\ w \ar[r]^{f^m_w} & x_w }$$ (uniquely), then set $$ \phi(h) := ( k, Fh) : (x_z, F f^m_z : F z \to Fx_z = \M x_z) \to (x_w, F f^m_w : Fw \to Fx_w ) . $$ The uniqueness of the completion ensures that this respects compositions.

\noindent {\bf Step 4.}  We construct an invertible $\CFG / \LogSch$ $2$-morphism $\eta : \phi \psi \to \Id_{ (\Z^m, \M)}$ as follows.  Note \be \phi \psi(x, f : X \to \M x) & = & \phi(z_x) \\ & = & (x_{z_x}, F f^m_{z_x} : F z_z \to F x_{z_x} ) , \ee so define $$\eta(x,f : X \to \M x) : \phi \psi(x,f : X \to \M x) \to (x, f : X \to \M x)$$ to be $$(k, \Id) : (x_{z_x}, F f^m_{z_x} : F z_z \to F x_{z_x} ) \to (x, f : X \to \M x) $$ where $k$ is the unique map completing the $\Z$ diagram $$ \xym@C+15pt{ z_x \ar@{=}[d] \ar[r]^-{f^m_{z_x}} & x_{z_x} \ar@{.>}[d]^k \\ z_x \ar[r]^-{f_x} & x } $$ which then lies over the $\LogSch$ diagram $$ \xym@C+15pt{ F z_x \ar@{=}[d] \ar[r]^-{F f^m_{z_x}} & F x_{z_x} \ar[d]^{Fk} \\ X \ar[r]^-f & \M x } $$ (such a unique completion $k$ exists because $x \in \Z^m$ is minimal and $\u{F}f^m_{x_z} = \u{F} f_x = \Id$; we don't even need Lemma~\ref{lem:minimality1} here).  Note that the map $k$ (hence also the map $\eta(x,f: X \to \M x) = (\Id,k)$) is invertible because $x_{z_x} \in \Z^m$ is also minimal.

\noindent {\bf Step 5.} We construct an invertible $\CFG / \LogSch$ $2$-morphism $\theta : \psi \phi \to \Id_{\Z}$ as follows: Note \be \psi \phi(z) & = & \psi(x_z, F f^m_z : Fz \to F x_z) \\ & = & z_{x_z} \\ & = & \Dom (f_{x_z} : z_{x_z} \to x_z) , \ee where $f_{x_z}$ was chosen so that $F f_{x_z} = F f^m_z$.  Since $F$ is a groupoid fibration we may define $\theta(z) : \psi \phi(z) \to z$ to be the unique $\Z$ morphism $\theta(z) : z_{x_z} \to z$ lying over the identity of $F z_{x_z} = Fz$ and making the $\Z$ diagram $$ \xym{ z_{x_z} \ar[rr]^{\theta(z)} \ar[rd]_{f_{x_z}} & & z \ar[ld]^{f^m_z} \\ & x_z } $$ commute (here we use that $F$ is a groupoid fibration).  Note that $\theta(z)$ is invertible because it is a morphism in a fiber category of $F$ and $F$ is a groupoid fibration.

\section{Examples} \label{section:examples}  All of our examples will be taken from log geometry, as developed in \cite{KK}.  Recall that a \emph{prelog structure} on a scheme $X$ is a map of sheaves of (commutative, unital) monoids $\alpha_X : \M_X \to \O_X$ ($\O_X$ is viewed as a monoid under multiplication) on the \'etale site of $X$.  We work exclusively in the \'etale site, so $\O_{X}$ is the structure sheaf on the \'etale site, a \emph{point} is a point of the \'etale site, a neighborhood is an \'etale neighborhood, and so forth. A morphism of prelog structures is a morphism of sheaves of monoids over $\O_X$.    A prelog structure is called a \emph{log structure} iff $\alpha_X | \alpha_X^{-1}(\O_X^*) : \alpha_X^{-1}(\O_X^*) \to \O_X^*$ is an isomorphism.  By abuse of notation, we always view $\O_X^*$ as a submonoid of a log structure $\M_X$, and we call the quotient $ \M_X / \O_X^*$ the \emph{characteristic monoid} and denote it $\ov{\M}_X$.  A log structure is called \emph{integral} iff $\M_X$ injects into its group envelope $\M_X^{\rm gp}$.  The inclusion of log structures into prelog structures has a left adjoint $$ \M_X \mapsto \M_X^{\a} := \M_X \oplus_{\alpha_X^{-1} \O_X^*} \O_X^*,$$ called the \emph{associated log structure}.  If the structure map $\alpha_X$ for the prelog structure $\M_X$ is not obvious (as will be the case in \S\ref{section:logpoints}), then we will write $(\M_X,\alpha_X)^{\a}$ for clarity.  A \emph{chart} for a log structure $\M_X$ consists of a finitely generated monoid $P$ together with a map $h : P \to \M_X(X)$ such that the map of prelog structures from $\alpha_X \u{h} : \u{P} \to \O_X$ to $\M_X$ induces an isomorphism on associated log structures.  A log structure is called \emph{coherent} iff it locally has a chart, and \emph{fine} iff it is coherent and integral.  A fine log structure is called \emph{fs} (fine and saturated) if it has charts where the monoid $P$ is integral and has the property: for any $q \in P^{\rm gp}$ with $nq \in P$ for some $n \in \ZZ_{>0}$, $q$ is in $P$.  

A \emph{log scheme} is a pair consisting of a scheme $X$ and a log structure $\M_X$ on $X$.  We typically denote this data by the single symbol $X$, reserving $\u{X}$ for the underlying scheme.  A \emph{morphism of log schemes} $f : X \to Y$ consists of a morphism $\u{f} : \u{X} \to \u{Y}$ between the underlying schemes together with a morphism $f^\dagger : f^* \M_Y \to \M_X$ of log structures on $\u{X}$.  Here $f^* \M_Y$ is the log structure associated to the prelog structure $f^\sharp (f^{-1} \alpha_Y) : f^{-1} \M_Y \to \O_X$.\footnote{ We should probably write $\u{f}^{-1}$ or even $\u{f}_{\acute{e}t}^{-1}$, but in general we will avoid underlining symbols when they clearly refer to structures from the underlying scheme or its \'etale site.}  The cokernel of $f^*$ is called the \emph{relative characteristic monoid} of $f$ and is denoted $\ov{\M}_{X/Y}$.  The morphism $f$ is called \emph{strict} iff $f^\dagger$ is an isomorphism.  The fact that $f$ (as a morphism between \'etale ringed topoi) is a local morphism ensures that $\ov{f^* \M_Y} = f^{-1} \ov{\M}_Y$, so the stalk of $f^\dagger$ at a point $x$ of $X$ is a morphism $\ov{f}^\dagger_x : \ov{\M}_{Y,f(x)} \to \M_{X,x}$.  As long as $\M_X, \M_Y$ are integral log structures, $f$ is strict iff each map $\ov{f}^\dagger_x$ is an isomorphism.  The forgetful functor $X \to \u{X}$ from log schemes to schemes is a fibered category where a morphism $f : X \to Y$ is cartesian iff it is strict.  A morphism $f$ of integral log schemes is called \emph{integral} iff $\ov{f}^\dagger_x$ is a monic, integral morphism (of sharp, integral monoids; see [KK 4.1(iv)] for the definition of an integral morphism) for every point $x$ of $X$.  The property of being integral (or coherent, fine, fs, \dots) is stable under the pullback operation $\M_X \mapsto f^* \M_X$, hence the full subcategory of integral (or coherent, \dots) log schemes is also fibered over schemes via the same forgetful functor.

\subsection{Log curves} \label{section:logcurves}  In this section, $\LogSch$ denotes the category of fs log schemes, fibered over $\Sch$ via the forgetful map $X \mapsto \u{X}$ retaining only the underlying scheme.

\noindent {\bf Definition.} (F.~Kato) A $\LogSch$ morphism $f : X \to Y$ is called a \emph{log curve} iff $f$ log smooth and integral (see [KK 3.3, \S 4] for the definitions of log smooth and integral morphisms; the last two conditions imply that $\u{f}$ is flat [KK 4.5]), $\u{f}$ is finitely presented\footnote{In some applications, it is best to say ``proper," but our study will be local in nature, so we do not require this.}, and every geometric fiber of $\u{f}$ is reduced and of pure dimension one.  Log curves form a category $\Z$ where a morphism $(f : X \to Y) \to (f' : X' \to Y')$ is a cartesian diagram $$ \xym{ X \ar[r] \ar[d] & X' \ar[d] \\ Y \ar[r] & Y' } $$ in $\LogSch$.\footnote{For general maps of integral log schemes, formation of fibered products in the category of integral log schemes does not commute with $X \mapsto \u{X}$, but it does commute when one of the maps is integral, as will always be the case in any fiber product diagram we consider.}  The obvious forgetful functor $(f : X \to Y) \mapsto Y$ makes $\Z$ a category fibered in groupoids over $\LogSch$. 

For a map of log schemes $f : X \to Y$, recall [KK \S 3] that the sheaf of log differentials $\Omega_{X/Y}$ is the quotient of $$\Omega_{\u{X} / \u{Y}} \oplus (\O_X \otimes_{\ZZ} \M_X^{\rm gp}) $$ by the relations: \begin{enumerate} \item $(d\alpha_X(m),0)-(0,\alpha_X(m) \otimes m)$ for $m \in \M_X$ \item $(0,1 \otimes f^\dagger(n))$ for $n \in f^{-1} \M_Y$. \end{enumerate}  We write $[db, c \otimes m]$ for the class of $(db, c \otimes m) \in \Omega_{\u{X} / \u{Y}} \oplus (\O_X \otimes_{\ZZ} \M_X^{\rm gp})$ in $\Omega_{X/Y}$, and for $m \in \M_X$, we set $\dlog m := [0, 1 \otimes m] \in \Omega_{X/Y}$ as usual.  The first relation ensures that $\alpha_X(m) \dlog m = d \alpha_X(m)$, where $d \alpha_X(m) \in \Omega_{\u{X}/\u{Y}}$ and we suppress the natural map $\Omega_{\u{X}/\u{Y}} \to \Omega_{X/Y}$.  For any point $x$ of $X$, there is a natural surjection \bne{logdiffsurj} k(x) \otimes_{\O_{X,x}} \Omega_{X/Y,x} & \to & k(x) \otimes_{\ZZ} \ov{\M}_{X/Y}^{\rm gp} \\ \nonumber a \otimes [db, c \otimes m] & \mapsto & a \ov{c} \otimes \ov{m} \ene which we will make use of in the proof of Theorem~\ref{thm:logcurves}, and which is used in the proof of [KK 3.5].

For a map of monoids $Q \to P$, we write $\Omega_{P/Q}$ as abuse of notation for the $\ZZ[P]$ module whose associated quasi-coherent sheaf $\Omega_{P/Q}^{\sim}$ is $\Omega_{\Spec (P \to \ZZ[P]) / \Spec (Q \to \ZZ[Q])}$.  The map \be 1 \otimes \dlog : \ZZ[P] \otimes_{\ZZ}  (P^{\rm gp}/Q^{\rm gp}) & \to & \Omega_{P/Q} \ee is an isomorphism [KK 1.8].  In particular, for $\Delta = (1,1) : \NN \to \NN^2$, $\Omega_{\NN^2/ \NN}$ is freely generated by $\dlog (1,0)$  (or $\dlog (0,1)$).

For any log smooth map $X \to Y$, the sheaf $\Omega_{X/Y}$ is locally free of finite type [KK 3.10], and for a log curve it is easily seen to be invertible (i.e.\ of rank one) using, for example, the Chart Criterion for Log Smoothness [KK 3.5] (c.f.\ Lemma~\ref{lem:dim} below).

Let $f : X \to Y$ be a log curve.  Recall the following results of Kato \cite{FK}: \begin{enumerate} \item \label{nodal} The underlying curve $\u{f} :\u{X} \to \u{Y} $ is nodal.\footnote{Theorem~1.1 in \cite{FK}.} \item \label{relchar} The stalk of the relative characteristic monoid $\ov{\M}_{X/Y,x}$ at any point $x$ of $\u{X}$ is either $0$, $\NN$, or $\ZZ$.\footnote{Lemma~1.2 and the last sentence on Page 219 in \cite{FK}.} \item \label{smooth} $\u{f}$ is smooth near $x$ unless $\ov{\M}_{X/Y,x} \cong \ZZ$, in which case $x$ is a node in $\u{f}^{-1}(\u{f}(x))$.\footnote{Claims~1,2,3 in \cite{FK}.} \end{enumerate}

We will end up reproving all of these results below because we need slightly refined versions.  The proofs are basically the same as Kato's.\footnote{There are a couple of places where Kato claims that one can use the characteristic monoid of a fine log structure as a chart \'etale locally.  I want to point out that this is not true (in positive characteristic), though it does not significantly effect Kato's arguments.  For example, suppose $k$ is any field where $k^*$ is not divisible.  (This can occur even if $k$ is separably closed.  For example, the indeterminate $x$ is not divisible by $p$ in the separable closure $k$ of the transcendental extension $\FF_p(x)$.)  Then there is a nontrivial extension $$0 \to k^* \to E \to \ZZ \oplus \ZZ / n \ZZ \to 0$$ for some positive integer $n$.  Let $P$ be the (fine, sharp) submonoid of $\ZZ \oplus \ZZ / n \ZZ$ generated by $(1,0),(1,1)$.  Note $P^{\rm gp} = \ZZ \oplus \ZZ / n \ZZ$.  Define an extension of monoids $$0 \to k^* \to Q \to P \to 0$$ by pullback.  Define a log structure $Q \to k$ by the identity on $k^* \subseteq Q$ and by $q \mapsto 0$ if $q \in Q \setminus k^*$.  This is a monoid homomorphism since $P$ is sharp, and the above monoid extension is the characteristic sequence for this log structure, but one cannot find a chart $h : P \to Q$, for then $h^{\rm gp}$ would split the original nontrivial extension of abelian groups. }

\begin{rem} The fact that we work with fine, \emph{saturated} log structures is needed only to conclude that $\ov{\M}_{X/Y}$ is saturated.  One may instead work in the category of all fine log schemes and simply add this condition in the definition of a log curve.  This has the advantage of simplifying various proofs, but has the disadvantage of looking rather artificial.  In any case, we will never need to know that the base $Y$ of the log curve is saturated.  Of course some saturation hypotheses are need if log smooth curves are to be nodal.  For example, if $P \subseteq \NN$ is the (fine!) submonoid generated by $2,3$, then $\Spec (P \to \ZZ[P])$ is log smooth and integral over $\Spec \ZZ$, but not nodal (it is a cusp).  It may be interesting to study integral, log smooth curves without saturation hypotheses. \end{rem}

In the next several lemmas, we will be considering a monomorphism $h : Q \into P$ of sharp (unit-free) monoids.  In this situation, an element $p \in P$ is called $Q$-\emph{primitive} iff, whenever $p = p'+q$ for some $q \in Q, p' \in P$, we have $q=0$.  We let $I_Q \subseteq P$ be the ideal of $P$ generated by the non-zero elements of $Q$.  That is: \be I_Q & := & (Q \setminus \{ 0 \} )+P \\ & = & \{ q+p : q \in Q \setminus \{ 0 \}, p \in P \}. \ee  We say that $p \in P$ is \emph{nilpotent for} $h$ iff $p \notin I_Q$ but $np \in I_Q$ for some $n \in \ZZ_{\geq 1}$.  We will need the following results, the first of which is F.~Kato's \emph{Integral Splitting Lemma} (we include a different proof):

\begin{lem} \label{lem:is} [FK, \S 1] Let $Q \into P$ be an integral monomorphism of fine, sharp monoids.  Then every element $p \in P$ can be uniquely written $p = p'+q$ for a $Q$-primitive $p' \in P$ and $q \in Q$.  In particular, every element of $P/Q$ has a unique $Q$-primitive representative in $P$. \end{lem}

\begin{proof} For $p_1,p_2 \in P$, declare $p_1 \leq p_2$ iff there is a $p \in P$ such that $p_1+p=p_2$.  Since $P$ is sharp and finitely generated, there are no infinite strictly $\leq$-decreasing sequences in $P$, so for each $p \in P$, there is some $p'$ with the same image as $p$ in $P/Q$, such that $p'$ is $\leq$-minimal with respect to this property; this obviously implies that $p'$ is $Q$-primitive.   Since $p,p'$ have the same image in $P/Q$, it follows from integrality that we can write $p = p''+q, p'=p''+q'$ for some $p'' \in P, q,q' \in Q$.  Since $p'$ is $Q$-primitive, we must have $q'=0$, so $p' = p''$ and $p = p' +q$ is the desired expression and the same reasoning shows that this is the unique such expression. \end{proof}

We will also need to know that primitivity is ``stable under pushout" in the following sense:

\begin{lem} \label{lem:pushout} Let $h : Q \to P$ be an integral monomorphism of sharp, integral monoids, satisfying the conclusion of Lemma~\ref{lem:is} and let $f : Q \to R$ be any morphism to a sharp, integral monoid.  Then the pushout $P \oplus_Q R$ is a sharp, integral monoid, and the pushout map $R \to P \oplus_Q R$ is an integral monomorphism (with the same cokernel as $h$) satisfying the conclusion of Lemma~\ref{lem:is}.  Furthermore, for any $Q$-primitive $p \in P$, $[p,0] \in P \oplus_Q Q$ is $R$-primitive.  \end{lem}

\begin{proof}  The map $p \mapsto [p,0]$ induces an isomorphism on cokernels for formal reasons (direct limits commute amongst themselves).  Due to the integrality assumption on $h$, $P \oplus_Q R$ is integral, $R \to P \oplus_Q R$ is an integral morphism, and $P \oplus_Q R$ is the quotient of $P \oplus R$ by the monoidal equivalence relation $\sim$ where $(p,r) \sim (p',r')$ iff there are $q,q' \in Q$ with $$(p+q, r+f(q')) = (p'+q', r'+f(q)).$$  All the statements follow easily from this. \end{proof}

\begin{lem} \label{lem:monoids} Let $h : Q \hookrightarrow P$ be an integral monomorphism of fine, sharp monoids, without nilpotents.  Then: \begin{enumerate} \item \label{torsionfree} If $P$ saturated, then the quotient $P^{\rm gp} / Q^{\rm gp}$ is torsion free, and, at least when $P^{\rm gp} / Q^{\rm gp} \cong \ZZ$, the quotient $P/Q$ is saturated (hence is isomorphic to $0,\NN,$ or $\ZZ$). \item \label{splitting} If $P/Q \cong \NN$, then there is a unique $p \in P$ such that $(h,p) : Q \oplus \NN \to P$ is an isomorphism. \item \label{pushout} If $P/Q \cong \ZZ$, then there is a unique $q_0 \in Q$ and $p_1,p_{-1} \in P$ (unique up to $p_1 \leftrightarrow p_{-1}$) such that the diagram below is cocartesian.  $$\xymatrix@C+15pt{ \NN \ar[r]^-{q_0} \ar[d]_\Delta & Q \ar[d]^h \\ \NN \oplus \NN \ar[r]^-{p_1,p_{-1}} & P }$$  \end{enumerate} \end{lem}

\begin{proof} \eqref{torsionfree} Suppose $P^{\rm gp} / Q^{\rm gp}$ has nontrivial torsion, so there are $p_1,p_2 \in P$ such that $p_1-p_2 \notin Q^{\rm gp}$ but $np_1-np_2 \in Q^{\rm gp}$ for some $n \in \ZZ_{>0}$, i.e. $$np_1+q_1 = np_2+q_2$$ for some $q_1,q_2 \in Q$.  By Lemma~\ref{lem:is}, we can find $Q$-primitive elements $p_1', p_2'$ with the same image in $P/Q$ as $p_1,p_2$, so after possibly replacing $p_i$ with $p_i'$, we can assume $p_1,p_2$ are $Q$-primitive.  Again by Lemma~\ref{lem:is}, we can write \be np_1 & = & a_1+b \\ np_2 & = & a_2+b \ee for some $a_1,a_2 \in Q$, and some $Q$-primitive $b \in P$ (i.e.\ $b$ is the unique $Q$-primitive representative of the common image of $np_1, np_2$ in $P/Q$).  If $a_1,a_2=0$, then we have $n(p_1-p_2)=0$ and $n(p_2-p_1)=0$ in $P^{\rm gp}$, hence $p_1-p_2, p_2-p_1 \in P$ because $P$ is saturated.  But then $p_1-p_2 = 0$ because $P$ is sharp, and this certainly contradicts $p_1-p_2 \notin Q^{\rm gp}$.  So it must be that one of $a_1,a_2$, say $a_1$, is nonzero.  Then $np_1 = a_1+b$ is manifestly in $I_Q$.  Now, if $p_1$ were not in $I_Q$, then it would be nilpotent, so we can now assume that $p_1$ is in $I_Q$.  Then, since $p_1$ is $Q$-primitive, it must actually be that $p_1 \in Q$.  But then $np_1$ is certainly zero in $P/Q$, so we must have $b=0$, hence $np_2 = a_2$ is in $Q$, and, in particular, $np_2 \in I_Q$.  I claim that $p_2$ is not in $I_Q$, hence is nilpotent (a contradiction).  Indeed, if $p_2$ were in $I_Q$, then in fact $p_2$ would be in $Q$ because it is $Q$-primitive, and this would violate $p_1-p_2 \notin Q^{\rm gp}$.  

It remains to prove that $P/Q$ is saturated when $P^{\rm gp}/Q^{\rm gp} \cong \ZZ$.  Choose an identification $P^{\rm gp}/Q^{\rm gp} \cong \ZZ$, so we can view $P/Q$ as a submonoid of $\ZZ$.  Since its groupification is $\ZZ$, we must be able to find $m,n \in P/Q \subseteq \ZZ$ relatively prime.  Then we can find $a,b \in \NN$ such that $am-bn=\pm 1$.  Now, for $m \in P/Q \subseteq \ZZ$, let $p_m \in P$ denote the unique $Q$-primitive lift of $m$.  Then for any $t \in \NN$ we must have $tp_m = p_{mt}$, otherwise $p_m$ would be nilpotent.  Consider the element $ap_m-bp_n \in P^{\rm gp}$.  If we can show that this element is in $P$, then we are done because this element maps to $\pm 1 \in \ZZ$, so we would have $\pm 1 \in P/Q$ and any submonoid of $\ZZ$ containing $\pm 1$ is clearly saturated.  Now, since $P$ is saturated, it is enough to show that $m(ap_m-bp_n) \in P$.  For this, it suffices to show that $amp_m=bmp_n+p_m$, or, equivalently, $p_{amm} = p_{bmn}+p_m$.  By Lemma~\ref{lem:is} we can write $$p_{bmn} + p_m = p_{amm}+q$$ for some $q \in Q$, so it is enough to prove $q=0$.  But if $q$ were not zero, then $p_m$ would be nilpotent since $$(bn+1)p_m = bnp_m+p_m = p_{bmn}+p_m = p_{amm}+q.$$  

\eqref{splitting} By Lemma~\ref{lem:is}, for each $n \in \NN$, there is a unique $Q$-primitive $p_n \in P$ mapping to $n$ in $P/Q \cong \NN$.  Using Lemma~\ref{lem:is}, it is clear that the set map $n \mapsto p_n$ provides a splitting iff it is a monoid homomorphism (iff $p_n = np_1$ for all $n \in \NN$, in which case $p=p_1$ is as desired).  Again by Lemma~\ref{lem:is}, we can write  $$p_1 = p_1 + 0, \; 2p_1 = p_2+q_2, \; 3p_1 = p_3 + q_3, \; \dots$$ for unique $q_i \in Q$, so our map is a monoid homomorphism iff all these $q_i$ are zero.  If one of them isn't, then for some $n > 1$ we have $np_1 \in I_Q$, even though $p_1$ itself if not in $I_Q$ because $p_1$ is $Q$-primitive but not in $Q$.  The uniqueness of the $p$ yielding such a splitting is clear from Lemma~\ref{lem:is} because such a $p$ must be $Q$-primitive and map to $1 \in P/Q = \NN$ since $(0,1)$ certainly has these properties in $Q \oplus \NN$. 

\eqref{pushout} By Lemma~\ref{lem:is}, for each $n \in \ZZ$, there is a unique $Q$-primitive $p_n \in P$ mapping to $n$ in $P/Q \cong \ZZ$ and we can write $p_1+p_{-1} = q_0$ for a unique $q_0 \in Q$ since $0 \in P$ is certainly the unique $Q$-primitive mapping to $0 \in \ZZ$ (i.e.\ $p_0=0$).  Define set maps $\NN \to Q$ by $n \mapsto nq_0$ and $\NN^2 \to P$ by $(m,n) \mapsto p_m+p_{-n}$.  By the same argument as in the previous proof, we must have $p_n = n p_1$ and $p_{-n}=n p_{-1}$, otherwise $p_1$ (or $p_{-1}$) will be nilpotent.  The uniqueness of $q_0,p_1,p_{-1}$ follows from Lemma~\ref{lem:pushout} and the fact that $(1,0),(0,1) \in \NN^2$ are the unique $Q$-primitive lifts of $1,-1 \in Q/P$ (for an approproate indentification $Q/P \cong \ZZ$; the ambiguity $p_1 \leftrightarrow p_{-1}$ results only from this choice of identification). Now at least the diagram commutes: $nq_0=n(p_1+p_{-1})=p_n+p_{-n}$, so we just need to check that the natural map \begin{eqnarray*} Q \oplus_\NN (\NN \oplus \NN) & \to & P \\ {\rm [}q,(m,n)] & \mapsto & q+p_m+p_{-n} \end{eqnarray*} is an isomorphism.  Surjectivity is clear since, according to the theorem, any $p$ can be written $p=p_n+q$ for some $n \in \ZZ$.  For injectivity, suppose $q+p_{m}+p_{-n} = q'+p_{m'}+p_{-n'}$ in $P$.  Then in particular, they map to the same element of the cokernel so we must have $m-n=m'-n'$.  After possibly exchanging the primed and unprimed terms, we may assume $m' \leq m$, so that $m-m' \in \NN$.  Since the two ways $\NN$ maps into the pushout have to agree, we have \begin{eqnarray*} {\rm [}q',(m',n')] + (m-m') & = & [q',(m,n)]  \\ {\rm [}q,(m,n)]+(m-m') & = & [(m-m')q_0+q,(m,n)].\end{eqnarray*}  Now we have $$q'+p_m+p_{-n} = (m-m')q_0+q+p_m+p_{-n} $$ in $P$ so it follows from integrality of $P$ and the uniqueness part of Lemma~\ref{lem:is} that $q'=(m-m')q_0+q$ in $Q$.  Since we know that the pushout monoid is integral (because $h$ is integral), we conclude the equality $[q',(p_{m'},p_{n'})]=[q,(p_m,p_n)]$ in the pushout from the integrality of the element $m-m'$. \end{proof}

\begin{lem} \label{lem:dim} Let $f : X \to Y$ be a log smooth morphism of fine log schemes, $x$ a geometric point of $X$.  Then the rank of the finitely generated abelian group $\ov{\M}_{X/Y,x}^{\rm gp}$ is no larger than the dimension of the geometric fiber $\u{f}^{-1}(\u{f}(x))$. \end{lem}

\begin{proof} We can assume $\u{Y} = \Spec k$, $k$ algebraically closed, so we have a chart $Q \to \M_Y$ with $Q = \ov{\M}_{Y,\u{f}(x)}$.  By applying the Chart Criterion for log smoothness, we can find a monomorphism $Q \into P$ of fine monoids appearing in a chart for $f$ near $x$ such that an \'etale neighborhood of $x$ in $\u{X}$ is isomorphic to $k[P]/I_Q^k$ as in the previous proof.   But $\dim k[P]/I_Q^k$ equals the rank of $P^{\rm gp} / Q^{\rm gp}$, and this latter group surjects onto $\ov{\M}_{X/Y,x}^{\rm gp}$ since $P$ is a chart.   \end{proof}

Finally we are in a position to prove the basic structure theorem for log curves:

\begin{thm} \label{thm:logcurves} Let $f : X \to Y$ be a log curve, $x$ a point of $X$, $y := f(x)$.  Then: \begin{enumerate} \item \label{characteristic} $\ov{\M}_{X/Y,x}$ is isomorphic to $0,\NN$, or $\ZZ$.  \item \label{smooth} If $\ov{\M}_{X/Y,x} = 0$, then $f$ is strict near $x$ and $\u{f}$ is smooth near $x$. \item \label{N} If $\ov{\M}_{X/Y,x} = \NN$, then there is a unique $p \in \ov{\M}_{X,x}$ such that $$(\ov{f}^\dagger_x,p) : \ov{\M}_{Y,y} \oplus \NN \to \ov{\M}_{X,x} $$ is an isomorphism of monoids.  Furthermore, after possibly replacing $X$ with a neighborhood of $x$, there is a lift $\ov{p} \in \M_X(X)$ of $p$ such that $\dlog \ov{p}$ freely generates $\Omega_{X/Y}$, $$(f^\dagger,\ov{p}) : f^{-1} \M_Y \oplus \NN \to \M_X$$ induces an isomorphism on associated log structures, and \be \O_Y[t] & \to & \O_X \\ t & \mapsto & \alpha_X(\ov{p}) \ee is an \'etale map taking $x$ to the ``origin" in the fiber over $y$.  In particular, $\u{f}$ is smooth near $x$ and $f$ is strict away from the zero locus of $\alpha_X(\ov{p}) \in \O_X(X)$.  \item \label{Z} If $\ov{\M}_{X/Y,x} \cong \ZZ$, then there is a unique $q_x \in \ov{\M}_{Y,y}$, called the \emph{element smoothing} $x$, and elements $p_1,p_{-1} \in \ov{\M}_{X,x}$ (unique up to $p_1 \leftrightarrow p_{-1}$) such that $$\xymatrix@C+15pt{ \NN \ar[r]^-{q_x} \ar[d]_\Delta & \ov{\M}_{Y,y} \ar[d]^{\ov{f}^\dagger_x} \\ \NN^2 \ar[r]^-{p_1,p_{-1}} & \ov{\M}_{X,x} }$$ is a pushout diagram of monoids.  Furthermore, after possibly replacing $X,Y$ with neighborhoods of $x,y$, there are lifts $\ov{q} \in \M_Y(Y)$ of $q_x$ and $\ov{p}_1,\ov{p}_{-1} \in \M_X(X)$ such that $\ov{p}_1 + \ov{p}_{-1} = f^\dagger \ov{q}$, $$(f^\dagger, \ov{p}_1, \ov{p}_{-1}) : f^* \M_Y \oplus_{\NN} \NN^2 \to \M_X $$ induces an isomorphism on associated log structures, and such that \be \O_Y[x,y]/(xy-\alpha_Y(\ov{q})) & \to & \O_X \\ x & \mapsto & \alpha_X(\ov{p}_1) \\ y & \mapsto & \alpha_X(\ov{p}_{-1}) \ee is an \'etale map taking $x$ to the origin in the fiber over $y$.  In particular, $x$ is a node in its geometric fiber and $f$ is strict (and $\u{f}$ is smooth) away from the simultaneous zero locus of $\alpha_X(\ov{p}_1), \alpha_X(\ov{p}_{-1}) \in \O_X(X)$. \end{enumerate} \end{thm}

\begin{proof} For \eqref{smooth}, note that the condition $\ov{\M}_{X/Y,x} = \Cok \ov{f}^\dagger_x = 0$ implies that the integral monomorphism $\ov{f}^\dagger_x : \ov{\M}_{X,y} \to \ov{\M}_{X,x}$ (of fine, sharp monoids) must be an isomorphism (use Lemma~\ref{lem:is} for surjectivity), so $f^* \M_Y \to \M_X$ induces an isomorphism on stalks of characteristics at $x$, hence is an isomorphism near $x$ by standard smallness properties of fine log structures.  It is a standard fact [KK 3.8] that a strict morphism is log smooth iff the underlying map is log smooth.

The next step is to prove that the map $\ov{f}_x : \ov{\M}_{Y,y} \to \ov{\M}_{X,x}$ (which is an integral monomorphism because the log curve $f$ is an integral morphism) does not have nilpotents.  We can assume $\u{Y}$ is the spectrum of an algebraically closed field $k$ (e.g.\ the algebraic closure of the separably closed field $k(y)$), in which case we can take a chart $Q \to \M_{Y}(Y)$ with $Q := \ov{\M}_{Y,y}$ (i.e.\ $\M_Y = Q \oplus k^*$ is the log structure associated to $0 : Q \to k$).\footnote{Any integral log structure on an algebraically closed field is of this form because $k^*$ is a divisible abelian group.}  We can also freely replace $X$ with an affine \'etale neighborhood $\Spec A$ of $x$, so by the Chart Criterion for Log Smoothness [KK, 3.5], we can assume there is a chart $$ \xym@C+10pt{ P \ar[r]^-{h} & \M_X(X) \ar[r]^-{\alpha_X} & A \\ Q \ar[u] \ar[r] \ar@/_1pc/[rr]_0 & \M_Y \ar[u]_{f^\dagger} \ar[r] & k \ar[u]_{f}  } $$ for the morphism $f$ with $Q \into P$ monic and such that $$k \otimes_{\ZZ[Q]} \ZZ[P] \to A $$ is \'etale.  In particular, the ring $k \otimes_{\ZZ[Q]} \ZZ[P]$ must be reduced, since $A$ is reduced because we assume the geometric fibers of a log curve are reduced and $\Spec A$ is \'etale over such a fiber.  Since $h$ is a chart, the natural map $$P \to \ov{\M}_{X,x} = P / \{ p \in P : \alpha_X(h(p)) \in \O_{X,x}^* \}$$ is a surjection of monoids under $Q$, so if there were some $p \in \ov{\M}_{X,x}$ not in \be I_Q' & := & (Q \setminus \{ 0 \}) + \ov{\M}_{X,x} \ee with $np \in I_Q'$ for some $n > 1$, then any lift $\ov{p} \in P$ of $p$ would be an element of $P$ not in \be I_Q & := & (Q \setminus \{ 0 \})+P \ee with $nm \in I_Q$.  But this would violate the fact that $k \otimes_{\ZZ[Q]} \ZZ[P]$ is reduced since this ring is the quotient of $k[P]$ by the ideal $I_Q^k$ consisting of those $\sum_i a_i[p_i]$ in $k[P]$ where $p_i \in I_Q$ whenever $a_i \in k^*$, so $[\ov{p}]$ would be a nontrivial nilpotent in $k[P]/I_Q^k$.  This proves the claim.

With the exception of the ``furthermore" statements in \eqref{N} and \eqref{Z}, all the statements now follow immediately from Lemma~\ref{lem:dim} and Lemma~\ref{lem:monoids}, applied to the nilpotent-free integral monomorphism of fine, sharp monoids $\ov{f}^\dagger_x : \ov{\M}_{Y,y} \to \ov{\M}_{X,x}$.  

For the ``furthermore" in \eqref{N}, note that we can certainly find, after possibly replacing $X$ with a neighborhood of $x$, some $\ov{p} \in \M_X(X)$ lifting $p \in \ov{\M}_{X,x}$.  The map $(f^\dagger,\ov{p}) : f^* \M_Y \oplus \NN \to \M_X$ becomes an isomorphism on associated log structures after possibly shrinking $X$ since it is an isomorphism on characteristics at $x$. Since the class of $p$ generates $\ov{\M}_{X/Y,x}^{\rm gp} \cong \ZZ$, the natural surjection $$k(x) \otimes_{\O_{X,x}} \Omega_{X/Y,x} \to k(x) \otimes_{\ZZ} \ov{\M}_{X/Y,x}^{\rm gp} $$ implies that $1 \otimes \dlog \ov{p}_x$ is a $k(x)$ basis for the one-dimensional $k(x)$ vector space $k(x) \otimes_{\O_{X,x}} \Omega_{X/Y,x}$.  But $\Omega_{X/Y}$ is an invertible sheaf, so this implies that, after possibly shrinking $X$, $\dlog p$ is a basis for $\Omega_{X/Y}$.  It remains to prove that $\u{f}$ is smooth near $x$.  This is local, so we can assume we have a chart $Q \to \M_Y(Y)$.  Then $(p,f^\dagger) : \NN \oplus Q \to \M_X(X)$ induces an isomorphism on stalks of characteristics at $x$, so we can assume, after possibly shrinking $X$, that it is a chart for $\M_X$.  Then we have a diagram $$ \xym{ X \ar[d]_f \ar[r] & \Spec (\NN \oplus Q \to \ZZ[\NN \oplus Q]) \ar[d] \\ Y \ar[r] & \Spec (Q \to \ZZ[Q]) } $$ where the horizontal arrows are strict.  Furthermore, since $\dlog p$ freely generates $\Omega_{X/Y} \cong \O_X$ and $\dlog (1,0)$ freely generates $\Omega_{\NN \oplus Q / Q}$, we conclude that the induced map from $X$ to $$Y \times_{\Spec (Q \to \ZZ[Q])} \Spec (Q \oplus \NN \to \ZZ[Q \oplus \NN]) = Y \times \Spec (\NN \to \ZZ[\NN])$$ is log \'etale (since it induces an isomorphism on sheaves of log differentials) and strict (by the 2-out-of-3 property of strict morphisms), hence the underlying morphism $\u{X} \to \u{Y} \times \AA^1$ is \'etale, hence the result.  

The ``furthermore" in \eqref{Z} is a jazzed up version of the same argument.  After shrinking $X$ and $Y$, we can assume we have lifts $\ov{q}, \ov{p}_1, \ov{p}_{-1}$ as indicated.  Since we know $\ov{f}^\dagger_x(q_x) = p_1+p_{-1}$, we can arrange that $$ \ov{p}_1 + \ov{p}_{-1} = f^\dagger \ov{q} $$ after possibly shrinking $X$ again and adjusting our choice of lift $\ov{p}_1$ by a unit.  The induced map $(f^\dagger,\ov{p}_1,\ov{p}_{-1}) : f^* \M_Y \oplus_{\NN} \NN^2 \to \M_X$ is an isomorphism after possibly shrinking $X$ because it induces an isomorphism on characteristics at $x$.  Since the image of $\ov{p}_1$ in $\M_{X/Y,x}^{\rm gp} \cong \ZZ$ is generator, we find, exactly as above, that $\dlog \ov{p}_1$ is a basis for $\Omega_{X/Y}$ after possibly shrinking $X$.  We can assume, as above, that we have a chart $Q \to \M_Y(Y)$ and that $(\ov{p}_1, \ov{p}_{-1}, f^\dagger) : \NN^2 \oplus_{\NN} Q \to \M_X(X)$ is a chart for $\M_X$.  We thus obtain a diagram $$ \xym@C+10pt{ X \ar[d]_f \ar[r]^-g & \Spec (Q \oplus_{\NN} \NN^2 \to \ZZ[Q \oplus_{\NN} \NN^2]) \ar[d] \\ Y \ar[r] & \Spec (Q \to \ZZ[Q]) } $$ where the horizontal arrows are strict and $$g^* \Omega_{ Q \oplus_{\NN} \NN^2 / Q} \to \Omega_{X/Y} $$ is an isomorphism (since $\dlog (0,1,0)$ is a basis for the former and $\dlog \ov{p}_1$ is a basis for the latter), hence the map from $X$ to $$Y \times_{\Spec (Q \to \ZZ[Q])} \Spec (Q \oplus_{\NN} \NN^2 \to \ZZ[Q \oplus \NN]) = Y \times_{\Spec (\NN \to \ZZ[\NN])} \Spec (\NN^2 \to \ZZ[\NN^2]) $$ is log \'etale and strict, so the underlying morphism of schemes $$\u{X} \to \u{Y} \times_{\AA^1} \AA^2 = \Spec_Y \O_Y[x,y]/(xy-\alpha_Y(\ov{q}))$$ is \'etale.     \end{proof}

\noindent {\bf Definitions.}  Let $f :X \to Y$ be a log curve, $y$ a point of $\u{Y}$, $N_y$ the set of nodes in the geometric fiber $\u{f}^{-1}(y)$ (equivalently, $N_y = \{ x \in \u{f}^{-1}(y) : \ov{\M}_{X/Y,x} \cong \ZZ \}$).  Define $q : N_y \to \ov{\M}_{Y,y}$ by mapping $x \in N_y$ to the element $q(x) \in \ov{\M}_{Y,y}$ smoothing $x$.  Let $P_y$ be the free monoid on $N_y$ and write $q : P_y \to \ov{\M}_{Y,y}$ as abuse of notation for the map corresponding to $q$ under the adjunction \be \Hom_{\Mon}(P_y, \ov{\M}_{Y,y}) & = & \Hom_{\Sets}(N_y,\ov{\M}_{Y,y}). \ee  Call $f : X \to Y$ \emph{basic} iff $q : P_y \to \ov{\M}_{Y,y}$ is an isomorphism for every $y$.

\begin{lem} \label{lem:strictminimal}  For any morphism of log curves $$ \xym{ X \ar[r] \ar[d] & X' \ar[d] \\ Y \ar[r]^f & Y' } $$  with $X' \to Y'$ basic, $X \to Y$ is basic iff $f$ is strict. \end{lem}

\begin{proof}  The fact that the square is cartesian means that, for any point $y$ of $\u{Y}$, the monoid $P_y$ is naturally identified with $P_{f(y)}$.  The result now follows from the fact that a morphism of integral log structures is an isomorphism iff it is an isomorphism on characteristics at each point. \end{proof}

The rest of the section is devoted to proving that basic log curves are the same thing as minimal log curves.

\begin{lem} \label{lem:basicimpliesminimal} A basic log curve is minimal. \end{lem}

\begin{proof}  Consider a solid cartesian diagram \bne{dia} & \xym@R-15pt{ X'' \ar@{.>}[rr]^h \ar[dd] & & X \ar[dd]  \\ & X' \ar[lu]^a \ar[ru]_b \ar[dd] \\ Y'' \ar@{.>}[rr]^<<<<<<<k & & Y \\ & Y' \ar[ru]_j \ar[lu]^i } \ene of fs log schemes where the vertical arrows are log curves, $\u{i} = \u{j} = \Id_{\u{Y}}$, and $X \to Y$ is basic.  We must prove that there are unique maps $k,h$ of fs log schemes completing the diagram (the ``new" square is then automatically cartesian).  Since the solid squares are cartesian and $\u{i} = \u{j} = \Id$, we can assume, after passing to an isomorphic solid diagram if necessary, that $\u{a} = \u{b} = \Id_{\u{X}}$ and that the the map on schemes underlying each of the three vertical arrows is the same map $\pi : \u{X} \to \u{Y}$. The maps $k,h$ must satisfy $\u{k} = \Id_{\u{Y}}$, $\u{h}=\Id_{\u{X}}$, so it equivalent to prove that there is a unique pair $(k^\dagger, h^\dagger)$ consisting of a map $k^\dagger : \M_Y \to \M_{Y''}$ of log structures on $\u{Y}$completing \bne{dia2} & \xym{ \M_{Y''} \ar[rd]_{i^{\dagger}} & &  \M_Y \ar@{.>}[ll]_{k^{\dagger}}  \ar[ld]^{j^{\dagger}} \\ & \M_{Y'} } \ene  and a map $h^\dagger : \M_X \to \M_{X''}$ completing the diagram \bne{dia3} & \xym@R-10pt{ \M_{X''} \ar[rd]_{a^\dagger} & &  \M_{X} \ar@{.>}[ll]_{h^\dagger}  \ar[ld]^{b^\dagger} \\ & \M_{X'}  \\ \M_{Y''} \ar[uu] \ar[rd] & & \M_Y \ar[uu] \ar[ld] \ar[ll]_<<<<<<<{k^\dagger} \\ & \M_{Y'} \ar[uu] } \ene (suppressing notation for $\pi^*$ in the bottom triangle) of log structures on $\u{X}$.  Since the squares in the solid diagram are cartesian, the three log curves we could use to define $N_y$ yield the same result (up to a natural bijection which we suppress).

We first argue that there is at most one such pair.  Consider a point $y$ of $Y$.  Since $X \to Y$ is basic, $P_y \to \ov{\M}_{Y,y}$ is an isomorphism, so $\ov{k}^\dagger_y : \ov{\M}_{Y,y} \to \ov{\M}_{Y'',y}$ is uniquely determined by the elements $\ov{k}^\dagger_y(q_x) \in \ov{\M}_{Y'',y}$ as $x$ runs over the nodes in the geometric fiber of $\pi$ over $y$.  But the completed left square (hence every square) in the diagram $$ \xym@C+10pt{ \ov{\M}_{X'',x} & \ar@{.>}[l]_-{\ov{h}^\dagger_x}  \ov{\M}_{X,x} & \ar[l]_-{p_{-1},p_1}  \NN^2 \\ \ov{\M}_{Y'',y} \ar[u] & \ar@{.>}[l]_-{\ov{k}_y^\dagger} \ov{\M}_{Y,y} \ar[u] & \ar[l]_-{q_x} \NN \ar[u]_\Delta }  $$ must be cocartesian (the right pushout square here is the one from Theorem~\ref{thm:logcurves}\eqref{Z}), so by the uniqueness statement in Theorem~\ref{thm:logcurves}\eqref{Z}, it must be that $\ov{k}^\dagger_y(q_x)$ is the element $q_x \in \ov{\M}_{Y'',y}$ smoothing $x$.  This proves that any two completions of \eqref{dia2} must agree on the level of characteristics, but then they must agree because $i^\dagger$ is invertible on units.  The fact that there can then be at most one $h^\dagger$ completing \eqref{dia3} is proved in much the same way (establish the uniqueness of $\ov{h}^\dagger_x$ based on appeal to the appropriate case of Theorem~\ref{thm:logcurves}). 

Now that we know there is at most one completion, we can construct $k^{\dagger}$ \'etale locally near a point $y$ of $\u{Y}$.    After possibly replacing $\u{Y}$ with a small enough \'etale neighborhood of $y$, we can assume $P_y \to \ov{\M}_{Y,y}$ lifts to a chart $s : P_y \to \M_Y(\u{Y})$ (since $X \to Y$ is basic) and $P_y \to \ov{\M}_{Y'',y}$ lifts to a map $t : P_y \to \M_{Y''}(\u{Y})$.  Now, since $s$ is a chart, we could use the map on associated log structures induced by $t$ as our $k^\dagger$ if only we knew that $ i^{\dagger}(\u{Y}) t = j^{\dagger}(\u{Y})s$.  But we do know these maps induce the same map on characteristics at $y$ (namely the natural map $P_y \to \ov{\M}_{Y',y}$), so after possibly shrinking to a smaller neighborhood of $y$ and adjusting the choice of lift of $q_x \in \ov{\M}_{Y'',y}$ to $t(x) \in \M_{Y''}(\u{Y})$ by a unit, we can arrange the desired equality.  

Next we have to construct the completion $h^\dagger$ near a point $x$ of $X$.  If the common relative characteristic monoid $$\ov{\M}_{X''/Y'',x} \cong \ov{\M}_{X'/Y',x} \cong \ov{\M}_{X/Y,x}$$ is zero, then by Theorem~\ref{thm:logcurves}\eqref{smooth} the vertical arrows of \eqref{dia3} are isomorphisms near $x$, so this is easy.  If the common relative characteristic at $x$ is $\NN$ or $\ZZ$, then that theorem says that, at least on characteristics at $x$, the diagram \eqref{dia3} is the sum of the diagram \eqref{dia2} previously dealt with and $\NN$, or the pushout of \eqref{dia2} along $\Delta: \NN \to \NN^2$.  We will explain how to complete \eqref{dia3} in the more difficult case where the relative characteristic is $\ZZ$, and leave the easier case of relative characteristic $\NN$ to the reader.  Set $y := \pi(x)$.  After shrinking $\u{Y}$ if necessary, we can assume we have a lift $h : P_y \to \M_Y(\u{Y})$ of the natural isomorphism $P_y \to \ov{\M}_{Y,y}$ and that $h$ is a chart for $\M_Y$.  Let $\ov{h} : P_y \to \M_X(\u{X})$ be the composition of $h$ and $\M_Y(\u{Y}) \to \M_X(\u{X})$.  After shrinking again if necessary, we can assume that we have lifts $\ov{p}_1, \ov{p}_{-1} \in \M_X(\u{X})$ of the distinguished elements $p_1,p_{-1} \in \ov{\M}_{X,x}$ as in Theorem~\ref{thm:logcurves}\eqref{Z} so that $$\ov{p}_1+\ov{p}_{-1} = \ov{h}(q_x), $$ where $q_x \in P_y$ is slight abuse of notation for the inverse of the element $q_x \in \ov{\M}_{Y,y}$ smoothing $x$ under the isomorphism $P_y \to \ov{\M}_{Y,y}$.  From the universal property of pushouts we then obtain a map $$s := (\ov{h}, \ov{p}_1, \ov{p}_2) : \ov{\M}_{X,x} \cong P_y \oplus_{\NN} \NN^2 \to \M_X(\u{X}),$$ which we can assume is a chart after shrinking $\u{X}$ if necessary.  Let $\ov{k} : P_y \to \M_{X''}(\ov{X})$ be the composition of $h$, $k^{\dagger}(\u{Y})$, and $\M_{Y''}(\u{Y}) \to \M_{X''}(\u{X})$.  After another shrinking, we can find lifts $t_1, t_{-1} \in \M_{X''}(\u{X})$ of the distinguished elements $p_1,p_{-1} \in \ov{\M}_{X'',x}$ and assume that they satisfy $$t_1+t_{-1} = \ov{k}(q_x),$$ hence we obtain a map $$t := (\ov{k},t_1,t_2) :  \ov{\M}_{X,x} \cong P_y \oplus_{\NN} \NN^2 \to \M_{X''}(\u{X}).$$  If we knew $ a^\dagger(\u{X})t = b^\dagger(\u{X})s$ then we could obtain the desired completion $h^\dagger$ of \eqref{dia3} as the map on associated log structures associated to $t$ using the fact that $s$ is a chart.  Notice that, by construction, we do have the desired equality $ a^\dagger(\u{X})t = b^\dagger(\u{X})s$ on the submonoid $P_y \subseteq \ov{\M}_{X,x}$, so the only issue is that $a^\dagger(\u{X})(t_i)$ may not be equal to $b^\dagger(\u{X})(\ov{p}_i)$, but they do have the same image in the characteristic $\ov{\M}_{X',x}$ (namely the distinguished element $p_i$), and we do know that $$a^\dagger(\u{X})(t_1 + t_2) = b^\dagger(\u{X})(\ov{p}_1+\ov{p}_2) $$ (because this element comes from the submonoid $P_y$), so after shrinking again, we can find a unit $u \in \O_{\u{X}}^*(\u{X})$ and replace $t_1$ with $u t_1$, $t_2$ with $u^{-1} t_2$ to arrange the desired equality.\footnote{We are mixing additive and multiplicative notation for monoids here.}   \end{proof}

\begin{lem} \label{lem:enoughbasics} For any log curve $f :X \to Y$, there is a basic log curve $f' : X' \to Y'$ and a morphism of log curves $$ \xym{ X \ar[r] \ar[d] & X' \ar[d] \\ Y \ar[r]^g & Y' } $$ with $\u{g} = \Id$. \end{lem}

\begin{proof} It suffices to construct this locally on $\u{Y}$ since the locally constructed basic log structures will glue uniquely on overlaps in light of the previous lemma.  Consider a point $y$ of $Y$.  Let $x_1,\dots,x_n$ be the singular (i.e.\ nodal) points in the geometric fiber of $\u{f}$ over $y$.  Let $x_{n+1},\dots,x_{n+m}$ be the points of the geometric fiber of $\u{f}$ over $y$ where the relative characteristic $\ov{\M}_{X/Y}$ is $\NN$ (one can see that there are only finitely many such points by using Theorem~\ref{thm:logcurves}\eqref{N}).  After possibly replacing $Y$ with a (strict \'etale) neighborhood of $y$, we can assume: \begin{enumerate} \item There are lifts $\ov{q}_i \in \M_Y(Y)$ of the elements $q_i \in \ov{\M}_{Y,y}$ smoothing $x_i$ for $i=1,\dots,n$. \item There are neighborhoods $U_i$ of $x_i$ and lifts $\ov{p}_{i,1}, \ov{p}_{i,-1} \in \M_X(U_i)$ of the distinguished elements $p_{i,1},p_{i,-1} \in \ov{\M}_{X,x_i}$ as in Theorem~\ref{thm:logcurves}\eqref{Z} for $i=1,\dots,n$ and lifts $\ov{p}_i \in \M_X(U_i)$ of the distinguished element $p_i \in \ov{\M}_{X,x_i}$ as in Theorem~\ref{thm:logcurves}\eqref{N} for $i=n+1,\dots,n+m$.  \item $f$ is strict (in particular $\u{f}$ is smooth) outside of the union of the $U_i$ and on the intersection $U_i \cap U_j$ for $i \neq j$. \end{enumerate}  Let $q := (\ov{q}_1, \dots, \ov{q}_n) : \NN^n \to \M_Y(Y)$, and consider the (fine) log structure $\M_{Y'}$ on $\u{Y}$ associated to the prelog structure $\alpha_Y q : \NN^n \to \O_Y(Y)$.  Note $\M_{Y'}$ maps to $\M_Y$ via $q^{\a}$.  For $i=1,\dots,n$, define $\M_{X'}|U_i$ to be the log structure associated to the composition of $$ h_i := ( f^\dagger q, \ov{p}_{i,1} , \ov{p}_{i,-1}) : \NN^n \oplus_{\NN} \NN^2 \to \M_X(U_i) $$ and $\alpha_X : \M_X(U_i) \to \O_X(U_i)$.  Note $\M_{X'}|U_i$ maps to $\M_X|U_i$ via $h_i^{\a}$ and $\u{f}^* \M_{Y'} | U_i$ maps to $\M_{X'}|U_i$ via $(\NN^n \to \NN^n \oplus_{\NN} \NN^2)^{\a}$.   For $i=n+1,\dots,n+m$, define $\M_{X'}|U_i$ to be the log structure associated to the composition of $$ h_i := (f^\dagger q, \ov{p}_i ) : \NN^n \oplus \NN \to \M_X(U_i) $$ and $\alpha_X : \M_X(U_i) \to \O_X(U_i)$.  Note $\M_{X'}|U_i$ maps to $\M_X|U_i$ via $h_i^{\a}$ and $\u{f}^* \M_{Y'} | U_i$ maps to $\M_{X'}|U_i$ via $(\NN^n \to \NN^n \oplus \NN)^{\a}$.  

We have defined maps of log schemes $$f'_i : (\u{U}_i, \M_{X'}|U_i) \to (\u{Y}, \M_{Y'})$$ for $i=1,\dots,m+n$.  Each map $f'_i$ is a log curve by the Chart Criterion and the \'etaleness statements of Theorem~\ref{thm:logcurves}.   Notice that the strict locus of $f'_i$ is the same as the strict locus of $f|U_i$, namely the complement of the simultaneous zero locus of $\alpha_X(\ov{p}_{i,1}), \alpha_X(\ov{p}_{i,-1})$ ($i=1,\dots,n$) or the nonzero locus of $\alpha_X(\ov{p}_{i})$ ($i=n+1,\dots,n+m$).  In particular, we have $\M_{X'}|U_i |U_{ij} = \u{f}^* \M_{Y'} = \M_{X'}|U_j|U_{ij} $ so we can glue the local $\M_{X'}|U_i$ to a global log structure $\M_{X'}$ which agrees with $\u{f}^* \M_{Y'}$ outside the non-strict loci of the $f'_i$.  Thus we obtain a log curve $f' : X' \to Y'$ fitting into a commutative diagram as in the statement of the lemma.  To check that the diagram is cartesian, it is enough to check that $$ \xym{ \ov{\M}_{Y',f(x)} \ar[r] \ar[d] & \ov{\M}_{Y,f(x)} \ar[d] \\ \ov{\M}_{X',x} \ar[r] & \ov{\M}_{X,x} } $$ is cocartesian at any point $x$ of $\u{X}$, which is clear from the description of the right vertical arrow from Theorem~\ref{thm:logcurves} and our explicit charts for the left vertical arrow.  

It remains to check that $f' : X' \to Y'$ is basic.  Consider a point $y'$ of $\u{Y}$.  Note that $\ov{q}_i$ is zero in $\ov{\M}_{Y',y}$ iff $\alpha_Y(\ov{q})_{y'} \in \O_{Y,y'}^*$ iff $\u{f}|\u{f}^{-1}(y') \cap U_i$ is smooth.  On the other hand, if $\alpha_Y(\ov{q})_{y'} \in \m_{y'}$, then the geometric fiber $\u{f}|\u{f}^{-1}(y')$ contains a unique singular point $x_i = x_i(y')$ given by the simultaneous vanishing of $\alpha_X(\ov{p}_{i,1}), \alpha_X(\ov{p}_{i,-1})$ (c.f.\ the \'etale map in Theorem~\ref{thm:logcurves}\eqref{Z}) and we have a pushout diagram:  $$ \xym@C+30pt{ \NN \ar[r]^-{\ov{q}_{i,y'}} \ar[d]_\Delta & \ov{\M}_{Y',y'} \ar[r] \ar[d] & \ov{\M}_{Y,y'} \ar[d] \\ \NN^2 \ar[r]^-{\ov{p}_{i,1,x_i}, \ov{p}_{i,-1,x_i}} & \ov{\M}_{X',x_i} \ar[r] & \ov{\M}_{X,x_i} } $$   In other words, the image of $\ov{q}_i \in \M_{Y}(Y)$ in $\ov{\M}_{Y',y'}$ \emph{is} the element smoothing $x_i$, and our chart description of $\M_{Y'}$ shows that \be \ov{\M}_{Y',y'} & = & \frac{\NN \langle \ov{q}_1, \dots, \ov{q}_n \rangle}{ \langle \ov{q}_i : \ov{q}_{i,y'} \in \O_{Y,y'}^* \rangle } \ee is the free monoid on these ``node smoothers."  \end{proof}

\begin{thm} A log curve is minimal iff it is basic.  The category of log curves satisfies the conditions \eqref{enoughminimals} and \eqref{strict}, hence is equivalent, by the Descent Lemma, to the category $(\X,\M)$ for some groupoid fibration $\X \to \Sch$ with log structure $\M : \X \to \Log$. \end{thm}

\begin{proof} We saw that basic log curves are minimal in Lemma~\ref{lem:basicimpliesminimal}.  It then follows from Lemma~\ref{lem:enoughbasics} and the last part of Proposition~\ref{prop:minimalityproperties} that minimal log curves are the same as basic log curves.  The last statement then follows from Lemma~\ref{lem:enoughbasics} and Lemma~\ref{lem:strictminimal}. \end{proof}

\subsection{Log points} \label{section:logpoints}  In this section we work in the category $\LogSch$ of all integral log schemes.  We do not care whether our log structures are quasi-coherent.  Let $P$ be a sharp (unit-free), integral monoid.  Let $G := \Spec \ZZ[P^{\rm gp}]$, regarded as a log scheme with trivial log structure, so \be G(Y) & := & \Hom_{\LogSch}(Y,G) \\ & = & \Hom_{\Sch}(\u{Y}, G) \\ & = & \Hom_{\Ab}(P^{\rm gp},\Gamma(Y,\O_Y^*)) \\  & = & \Hom_{\Mon}(P, \Gamma(Y,\O_Y^*)). \ee  (We always use the last description of $G(Y)$.)  Then $G$ acts on $\Spec (0: P \to \ZZ)$ in a natural manner inducing the trivial action on the characteristic monoid $P$.  For a log scheme $Y$, set \be Y_P & := & Y \times_{\Spec \ZZ} \Spec (0: P \to \ZZ).\ee  The log structure on $Y_P$ is given explicitly as \be \M_{Y_P} = \M_Y \oplus \u{P}, \ee where $\u{P}$ is the sheaf of locally constant functions to $P$ as usual.  The structure map is given by \be \alpha_Y \oplus 0 : \M_Y \oplus \u{P} & \to & \O_Y \\ (m,p) & \mapsto & \left \{ \begin{array}{lll} \alpha_Y(m), & & p=0 \\ 0, & & p \neq 0, \end{array} \right . \ee where $\alpha_Y : \M_Y \to \O_Y$ is the structure map for $Y$.  The action of $u \in G(Y)$ on $Y_P$ is given explicitly by \bne{action} \M_Y \oplus \u{P} & \to & \M_Y \oplus \u{P} \\ \nonumber (m,p) & \mapsto & (u(p)m, p) ,\ene where we view $u$ as a monoid homomorphism $u : P \to \Gamma(Y,\O_Y^*)$ and we view $\O_Y^*$ as a submonoid of $\M_X$ in writing $u(p)m$.  This is an action over $Y$ (under $\M_Y$) because $$(m,0) \mapsto (u(0)m,p)=(m,0).$$  This action clearly induces the trivial action on the relative characteristic monoid $\ov{\M}_{Y_P/Y} = \u{P}$ and is a faithful action (consider its behavior on $\O_Y^* \oplus \u{P} \subseteq \M_Y \oplus \u{P}$).  

\begin{rem} $G$ is the full automorphism group (scheme) of automorphisms of $\M_Y \oplus \u{P}$ under $\M_Y$ inducing the trivial action on the characteristic $\ov{\M}_Y \oplus \u{P}$. \end{rem}

The direct sum $\M_Y \oplus \u{P}$ is the categorical direct sum of $\M_Y$ and $\O_Y^* \oplus \u{P}$ in the category of log structures, but it is also the categorical \emph{product} in sheaves of monoids.  It is not, however, the product in the category of log structures; instead, for a log structure $\N_Y$ on $\O_Y$, a map $$f^\dagger = (a,h) : \N_Y \to \M_Y \oplus \u{P}$$ of sheaves of monoids is a map of log structures iff the following conditions are satisfied: \begin{enumerate} \item $\beta(n)=0$ whenever $h(n)$ is nonzero. \item $a | h^{-1}(0) : h^{-1}(0) \to \M_Y $ is a map of log structures. \end{enumerate}

\noindent {\bf Definitions.}  A $P$ \emph{log point} $\pi : Y' \to Y$ is a morphism of log schemes with $\u{\pi}=\Id$ such that there is a (strict) \'etale cover $U \to Y$ and an isomorphism $\phi : U_P \to Y' \times_Y U$ (over $U$) such that the ``transition function" $ (\pi_2^* \phi)^{-1} \pi_1^* \phi  \in \Aut( (U \times_Y U)_P / U \times_Y U )$ is in $G(U \times_Y U)$.  A \emph{morphism of log points} $$(h',h) :(Z' \to Z) \to (Y' \to Y)$$ is a cartesian diagram $$ \xym{ Z' \ar[r]^{h'} \ar[d] & Y' \ar[d] \\ Z \ar[r]^h & Y } $$ of log schemes such that $\u{h}=\u{h}'$ and, \'etale locally on $\u{Y}$, we can choose a trivialization $Y' \cong Y_P$ (inducing a trivialization $Z' \cong Z_P$) such that $h'$, viewed as a map $h' : Z_P \to Y_P$ via the trivializations, differs from the natural map $h_P : Z_P \to Y_P$ by the action of some $u \in G(Z)$. For a log scheme $X$, a $P$ \emph{log point in} $X$ is a log point $Y' \to Y$ together with a morphism $f : Y' \to X$.  A \emph{morphism} of $P$ log points in $X$ is a morphism of the underlying $P$ log points commuting with the maps to $X$.  The category of $P$ log points in $X$ is fibered in groupoids over $\LogSch$ via $$(Y' \to Y, f : Y' \to X) \mapsto Y.$$

The category of $P$ log points is equivalent to $B G$ essentially by definition, though we do not need this in what follows.

Since $G$ acts trivially on characteristics, there is a natural splitting $\ov{\M}_{Y'} = \ov{\M}_Y \oplus \u{P}$ for any log $P$ point $Y' \to Y$, though one can generally only find a splitting $\M_{Y'} \cong \M_Y \oplus \u{P}$ \'etale locally on $\u{Y}$. Given a log point $(Y' \to Y, f)$ in $X$, we now construct the following: \begin{enumerate} \item \label{NY} a log structure $\N_Y$ on $\u{Y}$, called the \emph{basic log structure}, and a map $z^\dagger : \N_{Y} \to \M_{Y}$ of log structures on $Y$. \item \label{NY'} a log structure $\N_{Y'}$ on $\u{Y}$ and a morphism $\pi^\dagger : \N_Y \to \N_{Y'}$ of log structures on $\u{Y}$ making $(\Id_{\u{Y}}, \pi^\dagger) : Z' \to Z$ a $P$ log point, where we set $Z := (\u{Y}, \N_Y)$, $Z' := (\u{Y}, \N_{Y'})$.  \item \label{logpointmap} A morphism $(z')^\dagger : \N_{Y'} \to \M_{Y'}$ making the square $$ \xym{ \N_Y \ar[r]^{z^\dagger} \ar[d]_{\pi^\dagger} & \M_Y \ar[d]^{\pi^\dagger} \\ \N_{Y'} \ar[r]^{(z')^\dagger} & \M_{Y'} } $$ cocartesian and making $$((\Id_{\u{Y}}, (z')^\dagger), (\Id_{\u{Y}}, z^\dagger)) : (Y' \to Y) \to (Z' \to Z)$$ a morphism of log points.  \item \label{g} A morphism $g^\dagger : f^* \M_X \to \N_{Y'}$ satisfying $f^\dagger = (z')^\dagger g^\dagger$, hence lifting the map of log points in \eqref{logpointmap} to a morphism of log points in $X$. \end{enumerate}

For \eqref{NY}, we \emph{choose}, locally on $\u{Y}$, a splitting $\M_{Y'} \cong \M_{Y} \oplus \u{P}$, and use it to decompose $f^\dagger$ as $$f^\dagger = (a,h) : f^{-1}\M_{X} \to \M_{Y} \oplus \u{P}. $$  The map $a : f^{-1} \M_X \to \M_Y$ is not necessarily a map of prelog structures when $f^{-1} \M_X$ is regarded as a prelog structure with the usual structure map $$ f^\sharp f^{-1} \alpha_X : f^{-1} \M_X \to \O_Y.$$  Never-the-less, it certainly is a map of prelog structures when $f^{-1} \M_X$ is regarded as a prelog structure via the ``unusual" map $\alpha_Y a : f^{-1} \M_X \to \O_Y$.  We let \be \N_Y & := & (f^{-1} \M_X, \alpha_Y a)^{\a} \ee be the log structure associated to this ``unusual" prelog structure, and we let $z^\dagger$ be the map $a^{\a} : \N_Y \to \M_Y$ of log structures induced by the map of prelog structures $a$.  If we choose a different local splitting $\M_{Y'} \cong \M_Y \oplus \u{P}$, then we obtain a different decomposition $f^\dagger = (a', h)$ related to the first decomposition by a commutative diagram $$ \xym{ f^{-1} \M_X \ar[r]^{(a,h)} \ar[rd]_{(a',h)} & \M_Y \oplus \u{P} \ar[d]^{(m,p) \mapsto (u(p)m,p)}_{\cong} \\ & \M_Y \oplus \u{P} }$$ where the vertical arrow is the action \eqref{action} of some $u \in G(Y)$; i.e.\ $a'(m)=u(h(m))a(m)$ for all $m \in f^{-1} \M_X$.  The log structures \be (f^{-1} \M_X, \alpha_Y a)^{\a} & = & f^{-1} \M_X \oplus_{ a^{-1} \O_Y^*} \O_Y^* \quad \quad {\rm and} \\ (f^{-1} \M_X, \alpha_Y a' )^{\a} & = & f^{-1} \M_X \oplus_{ (a')^{-1} \O_Y^*} \O_Y^* \ee are then naturally identified by the isomorphism \bne{gluingiso} (f^{-1} \M_X, \alpha_Y a)^{\a} & \to & (f^{-1} \M_X, \alpha_Y a' )^{\a} \\ \nonumber [m,v] & \mapsto & [m, u^{-1}(h(m)) v] . \ene  The map \eqref{gluingiso} is well-defined because $$ [m,u^{-1}(h(m))] = [0,a'(m)u^{-1}(h(m))] = [0,a(m)] $$ in $(f^{-1} \M_X, \alpha_Y a' )^{\a}$ (and its inverse is well-defined for similar reasons).  Note that the isomorphism \eqref{gluingiso} is a natural isomorphism of log structures over $\M_Y$ (it respects the maps $a^a, (a')^a$), so we can glue the locally defined $\N_Y$ and $z^\dagger : \N_Y \to \M_Y$ into a global log structure $\N_Y$ and map of log structures $z^\dagger : \N_Y \to \M_Y$.  

For \eqref{NY'} and \eqref{logpointmap}, we just define $\N_{Y'}$ locally in the presence of a splitting $\M_{Y'} \cong \M_Y \oplus \u{P}$ to be $\N_{Y'} := \N_Y \oplus \u{P}$, then we glue the locally defined $\N_{Y'}$ using the same transition function that was used to glue $\M_Y \oplus \u{P}$, so that the maps $a \oplus \Id : \N_Y \oplus \u{P} \to \M_Y \oplus \u{P}$, defined with the aid of a local splitting $\M_{Y'} \cong \M_Y \oplus \u{P}$ and corresponding decomposition $f^\dagger = (a,h)$, glue to form maps $\N_{Y'} \to \N_Y$ (these then clearly make the diagram in \eqref{logpointmap} a pushout since this is a local issue).  More precisely, if two trivializations of $\M_{Y'}$ are related by the action \eqref{action} for some $u \in G(Y)$, then the corresponding trivializations of $\N_{Y'}$ are related by the isomorphism \bne{gluingiso'} (f^{-1}\M_X, \alpha_Y a)^{\a} \oplus \u{P} & \to & (f^{-1}\M_X, \alpha_Y a')^{\a} \oplus \u{P} \\ \nonumber [m,v,p] & \mapsto & [m, u^{-1}(h(m)) u(p)v,p]. \ene  We define the map $g^\dagger : f^* \M_X \to \N_{Y'}$ locally in the presense of a trivialization $\M_{Y'} \cong \M_Y \oplus \u{P}$ and corresponding decomposition $f^\dagger = (a,h)$ to be given by \bne{localg} g^\dagger : f^* \M_X & \to & (f^{-1} \M_X, \alpha_Y a)^{\a} \oplus \u{P} \\ \nonumber [m,v] & \mapsto & [m,v,h(m)], \ene noting that $f^* \M_X = (f^{-1}\M_X, f^\sharp f^{-1} \alpha_X)^{\a}$.  The fact that the local definitions of $\N_{Y'} \to \M_{Y'}$ and $g^\dagger$ glue is summed up in the following commutative diagram:  $$ \xym@C+25pt@R-10pt{ & (f^{-1} \M_X, \alpha_Y a')^{\a} \oplus \u{P} \ar[r]^-{a' \oplus \Id} & \M_Y \oplus \u{P} \\ f^* \M_X \ar[ru]_{\eqref{localg}} \ar[rd]^{\eqref{localg}} \\ & (f^{-1} \M_X, \alpha_Y a)^{\a} \oplus \u{P} \ar[uu]^{\eqref{gluingiso'}}_{\cong} \ar[r]^-{a \oplus \Id} & \M_Y \oplus \u{P} \ar[uu]^{\eqref{action}}_{\cong} } $$

\begin{rem} It is clear from the local construction of the basic log structure $\N_Y$ that $\N_Y$ will be coherent (resp.\ fine, \dots) whenever $\M_X$ is coherent (resp.\ fine, \dots). \end{rem}

\noindent {\bf Definition.} A $P$ log point $(Y' \to Y,f)$ in $X$ is called \emph{basic} iff the map $\ov{z}^\dagger$ of \eqref{NY} is an isomorphism.

It is important to understand the difference between the log structure $f^* \M_X$ and the basic log structure $\N_Y$.  Both are constructed, at least locally, as log structures associated to prelog structures where the domain monoid is $f^{-1} \M_X$.  But the structure maps to $\O_Y$ are radically different.  For example, let us compare their characteristic monoids.  Notice that the submonoid $a^{-1} \O_Y^* \subseteq f^{-1} \M_X$ defined with the aid of a local splitting $\M_{Y'} \cong \M_Y \oplus \u{P}$ and corresponding decomposition $f^\dagger = (a,h)$ does not actually depend on the choice of local splitting, as is clear from the form of the action \eqref{action} relating two different choices of splitting.  Since $\N_Y = (f^{-1} \M_X, \alpha_Y a)^{\a}$ and formation of characteristic monoids commutes with $\a$, the quotient $f^{-1} \M_X / a^{-1} \O_Y^*$ is nothing but the characteristic monoid $\ov{\N}_Y$ of $\N_Y$.  On the other hand, the characteristic monoid of $f^* \M_X$ is the same as the characteristic monoid of the prelog structure $$(f^{-1} \M_X, f^\sharp f^{-1} \alpha_X) = (f^{-1} \M_X, \alpha_{Y'} f^\dagger).$$  Given a local decomposition $f^\dagger = (a,h)$, note that $m \in f^{-1} \M_X$  will map to a unit in $\O_Y$ iff $a(m) \in \O_Y^* \subseteq \M_Y$ \emph{and} $h(m) = 0$.  That is, $f^{-1} \ov{\M}_X$ is the quotient of $f^{-1} \M_X$ by the submonoid $\{ m \in a^{-1} \O_Y^* : h(m)=0 \}$, so the natural map $\pi_1 \ov{g}^\dagger : f^{-1} \ov{\M}_X \to \ov{\N}_X$ is surjective.   (Here $\pi_1$ is the natural projection from $\ov{\N}_{Y'} = \ov{\N}_Y \oplus \u{P}$ to $\ov{\N}_Y$.)  In particular, if $(Y' \to Y,f)$ is basic, then $$\pi_1 \ov{f}^\dagger = \ov{z}^\dagger \pi_1 \ov{g}^\dagger : f^{-1} \ov{\M}_X \to \ov{\M}_Y$$ must be surjective, but the converse does not hold in general.  

\begin{rem} In some simple situations, surjectivity of $\pi_1 \ov{f}^\dagger$ does imply that $(Y' \to Y,f)$ is basic.  For example, suppose the log structure $\M_X$ has $\ov{\M}_{X,x}$ isomorphic to $0$ or $\NN$ for every point $x$ of $X$, and suppose $\pi_1 \ov{f}^\dagger$ is surjective.  Then I claim $(Y' \to Y,f)$ is basic.  Our log structures are integral, so it is enough to show that $\ov{z}^\dagger_y : \ov{\N}_{Y,y} \to \ov{\M}_{Y,y}$ is an isomorphism.  Since $\pi_1 \ov{g}^\dagger_y : \ov{\M}_{X,f(y)} \to \ov{\N}_{Y,y}$ is surjective, $\ov{\N}_{Y,y}$ must be isomorphic to $0$ or $\NN$, and $\ov{z}^\dagger_y$ must be surjective (because of the factorization $\pi_1 \ov{f}^\dagger = \ov{z}^\dagger \pi_1 \ov{g}^\dagger$).  But $\ov{z}^\dagger_y$, like any map on stalks of characteristics of log structures, is a map between sharp monoids with trivial kernel, so when its domain is $0$ or $\NN$ it is surjective iff it is an isomorphism.  This does not hold for maps out of $\NN^2$, and indeed, $f^\dagger =(a,h) := ((1,1),(0,0)) : \NN^2 \to \NN \oplus \NN$ defines an $\NN$ log point $(Y' \to Y,f)$ in $X = \Spec (0 : \NN^2 \to k)$ with $Y = \Spec ( 0 : \NN \to k)$ where $\pi_1 \ov{g}^\dagger$ is $\Id : \NN^2 \to \NN^2$ and $\ov{z}^\dagger = \pi_1 \ov{f}^\dagger$ is $(1,1) : \NN^2 \to \NN$. \end{rem}

\begin{lem} Every $P$ log point $(Z' \to Z, f)$ in $X$ admits a morphism $(b',b)$ to a basic $P$ log point in $X$ with $\u{b}=\Id$.  Given a morphism $(b',b) :(Z' \to Z,b'f) \to (Y' \to Y,f)$ of $P$ log points in $X$ with $(Y' \to Y, f)$ basic, $(Z' \to Z, fh')$ is basic iff $b$ is strict. \end{lem}

\begin{proof} We proved the first statement as part of the main construction of this section.  For the second statement, the key point is just to check that the construction of the basic log structure $\N_Y$ commutes with base change, which should be fairly obvious: Since the morphism of log points is a cartesian diagram, a local splitting $\M_{Y'} \cong \M_Y \oplus \u{P}$  and decomposition $f^\dagger = (a,h)$ induces a local splitting $\M_{Z'} \cong \M_Z \oplus \u{P}$ so that $(b')^\dagger = b^\dagger \oplus \Id$, so we get a decomposition $(b')^\dagger b^{-1} f^\dagger = (b^\dagger b^{-1} a, b^{-1} h)$.  The basic log structure $\N_Z$ would then by given by $( b^{-1} f^{-1} \M_X, \alpha_Z b^\dagger b^{-1} a)^{\a}$.  But $\alpha_Z b^\dagger = b^\sharp b^{-1} \alpha_Y$, and $( b^{-1} f^{-1} \M_X , b^{\sharp} b^{-1} \alpha_Y a)^{\a}$ is just $b^*( f^{-1} \M_X, \alpha_Y a)^{\a}$, which is just $b^* \N_Y$.   Then we have a commutative diagram $$ \xym{ \N_Z \ar[d] & \ar[l]_{\cong} b^* \N_Y \ar[d] \\ \M_Z & \ar[l]_{b^\dagger} b^* \M_Y } $$ where the top horizontal arrow is the isomorphism just discussed, and the right vertical arrow is an isomorphism since $(Y' \to Y, f)$ is basic.  Evidently then, $b^\dagger$ is an isomorphism (i.e.\ $b$ is strict) iff $\N_Z \to \M_Z$ is an isomorphism (i.e.\  $(Z' \to Z,b'f)$ is basic). \end{proof}  

\begin{lem}  A basic $P$ log point in $X$ is minimal. \end{lem}

\begin{proof}  We must uniquely complete any diagram $$ \xym@R-10pt{ & X \\ W' \ar@{.>}[rr]^j \ar[ru]^{f_1} \ar[dd] & & Z' \ar[lu]_{f_2} \ar[dd] \\ & Y' \ar[lu] \ar[ru] \ar[dd] \\ W \ar@{.>}[rr]^<<<<<<<i & & Z \\ & Y \ar[lu]^{i_1} \ar[ru]_{i_2} } $$ of $P$ log points where $\u{i}_1=\u{i}_2=\Id$ and $(Z' \to Z, Z' \to X)$ is basic.   Note that all underlying scheme maps in this diagram are $\Id$, except possibly $\u{f}_1 = \u{f}_2$, but there is no harm in replacing $X$ with the common underlying scheme and the pullback log structure, so we can assume there is only one scheme $\u{X}$ involved here and that all scheme maps are $\Id$.  We must then uniquely complete the diagram \bne{logpointdiagram} & \xym@R-10pt{ & \M_X \ar[ld]_{f_1^\dagger} \ar[rd]^{f_2^\dagger} \\ \M_{W'}  \ar[rd]  & & \ar@{.>}[ll]_{j^\dagger} \ar[ld] \M_{Z'}   \\ & \M_{Y'}   \\ \M_{W} \ar[uu] \ar[rd]_{i_1^\dagger} & & \ar@{.>}[ll]_<<<<<<<{i^\dagger} \ar[ld]^{i_2^\dagger} \M_{Z} \ar[uu] \\ & \M_{Y} \ar[uu]  } \ene of log structures on $\u{X}$ \emph{with maps of log structures that are actually maps of log points}.  It suffices to show that there is a unique completion locally. By definition of $P$ log points, we can assume, after shrinking, that \eqref{logpointdiagram} takes the form  \bne{logpointdiagram2} & \xym@R-10pt{ & \M_X \ar[ld]_{(a_1,h)} \ar[rd]^{(a_2,h)} \\ \M_{W} \oplus \u{P}  \ar[rd]_{j_1^\dagger}  & & \ar@{.>}[ll]_{j^\dagger} \ar[ld]^{j_2^\dagger} \M_{Z} \oplus \u{P}   \\ & \M_{Y} \oplus \u{P}   \\ \M_{W} \ar[uu] \ar[rd]_{i_1^\dagger} & & \ar@{.>}[ll]_<<<<<{i^\dagger} \ar[ld]^{i_2^\dagger} \M_{Z} \ar[uu] \\ & \M_{Y} \ar[uu] } \ene where $j_1^\dagger$, $j_2^\dagger$ are given by \be (m,p) & \mapsto & (u_1(p) i_1^\dagger(m), p) \\ (m,p) & \mapsto & (u_2(p) i_2^\dagger(m), p) \ee respectively, for some $ u_1, u_2 : P^{\rm gp} \to \Gamma(\u{X}, \O_{\u{X}}^*) $.  By definition of ``basic" and the basic log structure, $a_2 : \M_X \to \M_Z$ induces an isomorphism on associated log structures (when $\M_X$ is viewed as a prelog structure via the ``unusual" structure map $\alpha_Z a_2$), so the data of the log structure map $i^\dagger$ making the lower triangle commute is the same as the data of a prelog structure map $i : \M_X \to \M_W$ satisfying $i_1^\dagger i = i_2^\dagger a_2$.  The commutativity of the solid square means \be u_1(h(m))i_1^\dagger(a_1(m)) & = & u_2(h(m)) i_2^\dagger (a_2(m)) \ee for $m \in \M_X$, so $i(m) := (u_1^{-1}u_2)(h(m)) a_1(m)$ will do the trick.  We can then complete the top triangle by setting $j^\dagger(m,p) := ((u_1^{-1}u_2)(p)i^\dagger(m),p)$.  To check the uniqueness of these completions, we first check uniqueness on the level of characteristics, where \eqref{logpointdiagram2} becomes $$ \xym@R-10pt{ & \ov{\N}_Z \ar[ld]_{(\ov{a}_1,h)} \ar[rd]^{(\ov{a}_2,h)} \\ \ov{\M}_{W} \oplus \u{P}  \ar[rd]_{\ov{i}_1^\dagger \oplus \Id}  & & \ar@{.>}[ll]_{\ov{j}^\dagger} \ar[ld]^{\ov{i}_2^\dagger \oplus \Id} \ov{\M}_{Z} \oplus \u{P}   \\ & \ov{\M}_{Y} \oplus \u{P}   \\ \ov{\M}_{W} \ar[uu] \ar[rd]_{\ov{i}_1^\dagger} & & \ar@{.>}[ll]_<<<<<{\ov{i}^\dagger} \ar[ld]^{\ov{i}_2^\dagger} \M_{Z} \ar[uu] \\ & \M_{Y} \ar[uu] } $$ and it is clear from the fact that $\ov{a}_2$ is an isomorphism that this diagram has a unique completion (namely $\ov{i}^\dagger = \ov{a}_2^{-1} \ov{a}_1$, $\ov{j}^\dagger = \ov{i}^\dagger \oplus \Id$).  Once we know the completion is unique on the level of characteristics, one sees easily that it is unique using the fact that $i^\dagger_1$, $j^\dagger_1$ are ``invertible on units". \end{proof}

\begin{thm} A $P$ log point in $X$ is minimal iff it is basic.  The category of $P$ log points in $X$ satisfies the conditions \eqref{enoughminimals} and \eqref{strict}, hence is equivalent, by the Descent Lemma, to the category $(\X,\M)$ for some groupoid fibration $\X \to \Sch$ with log structure $\M : \X \to \Log$. \end{thm}

\begin{proof} The proof is the same as the proof of Theorem~\ref{thm:logcurves}, using the lemmas above. \end{proof}

\end{document}